\documentclass[10pt,reqno]{amsart}

\usepackage{graphicx}
\graphicspath{ {./images/} }
\usepackage{xargs}
\usepackage{geometry}
\usepackage{fancyhdr} 
\usepackage{lastpage} 
\usepackage{extramarks} 
\usepackage[most]{tcolorbox} 
\usepackage{xcolor} 
\usepackage[hidelinks]{hyperref} 
\usepackage[permil]{overpic}
\usepackage{pict2e} 
\usepackage{listings}
\usepackage{amssymb}
\usepackage{amsthm}
\theoremstyle{plain}
\usepackage[
backend=biber,
style=alphabetic,
sorting=ynt
]{biblatex}
\usepackage{xargs}
\usepackage{float}
\usepackage{lipsum}
\usepackage{geometry}
\usepackage{fancyhdr} 
\usepackage{lastpage} 
\usepackage{extramarks} 
\usepackage[most]{tcolorbox} 

\usepackage{xcolor} 
\usepackage[hidelinks]{hyperref} 
\usepackage[permil]{overpic}
\usepackage{pict2e} 
\usepackage{listings}

\newtheorem*{theorem*}{Theorem}

\theoremstyle{notation}

\numberwithin{equation}{section}
\theoremstyle{plain}
\newtheorem{definition}{Definition}[section]
\newtheorem{theorem}[equation]{Theorem}
\newtheorem{claim}{Theorem}
\newtheorem{remark}[equation]{Remark}
\newtheorem{corollary}[equation]{Corollary}
\newtheorem{lemma}[equation]{Lemma}
\newtheorem{conj}[equation]{Conjecture}
\newtheorem{proposition}[equation]{Proposition} 
\providecommand{\keywords}[1]
{
  \small	
  \textbf{\textit{Keywords---}} #1
}

\begin{document}
\title{One dimensional dynamics and the Rössler Attractor}
%
%
\author{Eran Igra}

%
%
\begin{abstract}
The Rössler system is one of the most famous dynamical systems, mostly due to its numerically-observed attractor - which is generated by a fold mechanism. In this paper we state and prove a topological criterion for the existence of an attractor for the Rössler system - and then analyze the periodic dynamics of the non-wandering set by reducing the flow dynamics to a well-known one dimensional model: the quadratic family, $x^2+c$, $c\in[-2,\frac{1}{4}]$.

\end{abstract}
\maketitle    
\keywords{\textbf{Keywords} - The Rössler Attractor, Chaos Theory, Heteroclinic bifurcations, Topological Dynamics}
\section*{Introduction}         

Recall the Rössler system, first introduced in \cite{Ross76}:
\begin{equation} \label{Vect}
\begin{cases}
\dot{X} = -Y-Z \\
 \dot{Y} = X+AY\\
 \dot{Z}=B+Z(X-C)
\end{cases}
\end{equation}
With parameters $A,B,C\in\mathbf{R}^{3}$. Inspired by a taffy-pulling machine (see \cite{Ro83}), this dynamical system was originally introduced in 1976 by O.E. Rössler to model a suspended stretch-and-fold operation. By varying the $A,B,C$ parameters, Rössler numerically discovered that at $(A,B,C)=(0.2,0.2,5.7)$ the flow generates a chaotic attractor, which had the shape of a band stretched and folded on itself (see the illustration in Fig.\ref{original}). In more detail, in \cite{Ross76} Rössler observed that for $(A,B,C)=(0.2,0.2,5.7)$ the flow had an attracting, invariant set, whose first return map had the shape of a Smale horseshoe (see \cite{S}) - which he used to explain the seemingly complex dynamics of the flow.\\

\begin{figure}[h]
\centering
\begin{overpic}[width=0.4\textwidth]{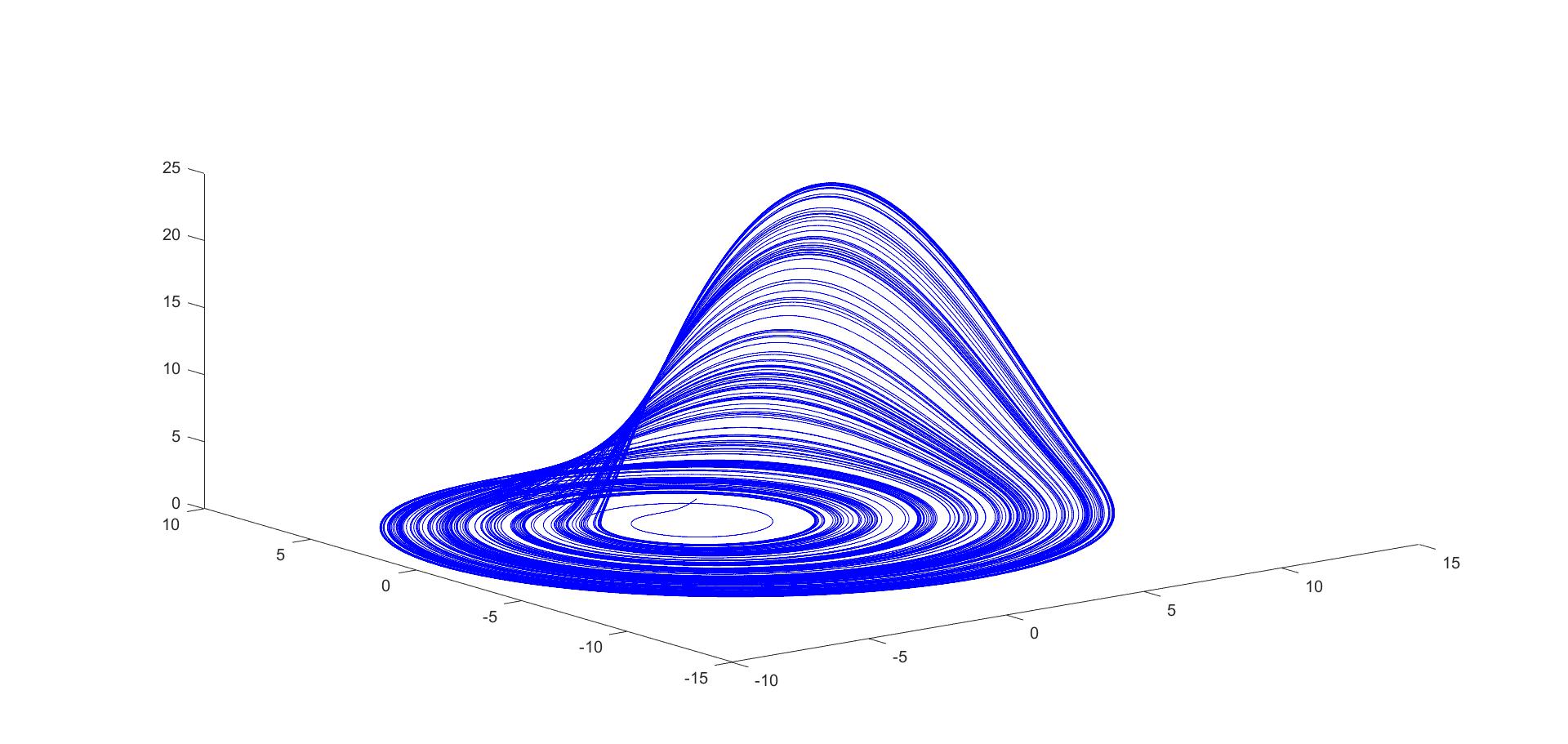}
\end{overpic}
\caption[Fig1]{The Rössler attractor at $(A,B,C)=(0.2,0.2,5.7)$}\label{original}
\end{figure}

Since its introduction in 1976, the Rössler system was the focus of many numerical studies - despite the simplicity of the vector field, the Rössler system gives rise to many non-linear phenomena, which are often linked to homoclinic bifurcations (see, for example, \cite{MBKPS}, \cite{BBS}, \cite{G}, \cite{BB2}, \cite{SR} and the references therein). One particular feature is that varying the parameters $A,B,C$ leads to a rise in the complexity of the system. In more detail, as the $A,B,C$ parameters are varied, more and periodic trajectories for the flow appear via period-doubling and saddle-node bifurcations, until they finally collapse to a chaotic attractor (see \cite{MBKPS}, \cite{BBS},\cite{WZ} and \cite{Ross76}). In addition, as observed numerically in \cite{SBM}, the Rössler system includes a period-doubling route to chaos and satisfies a form of Feigenbaum Universality with a constant $\delta\approx4.669$ - the same  as the constant originally observed for the Logistic Family, $\lambda x(1-x)$, $\lambda\in(0,4]$ (see \cite{Fei}).\\

In contrast to the large corpus of numerical studies, analytic results on the Rössler system are few. The existence of chaotic dynamics for the Rössler system in the original parameters considered by O.E. Rössler was first proven in \cite{Zgli97} and later on, also in \cite{XSYS03} - both with the aid of rigorous numerical methods. The existence of periodic orbits for some parameter values was proven in \cite{Pan}, while later on, in \cite{LiLl} the dynamics of the Rössler system at $\infty$ were studied by applying the Poincare sphere method. More recent results include \cite{CNV}, where the existence of an invariant torus (and its breakdown) was analyzed - as well as \cite{zgli1} and \cite{zgli2}, where the existence of infinitely many periodic trajectories was proven at specific parameter values, using rigorous numerical methods.\\

To our knowledge, to this date no study on the Rössler system ever attempted to study the existence of the Rössler attractor and analyze its dynamics - and it is precisely this gap this paper aims to address. Namely, in this paper we prove a topological condition, which, when satisfied by the Rössler system, implies the existence of an attracting invariant set for the flow, $A$. As we will see, $A$ strongly resembles the numerically observed attractor. Later on, inspired by \cite{zgli1} and \cite{zgli2} (and by the numerical studies cited above), by applying previous results obtained by the author in \cite{I} we prove the dynamics of the Rössler system on its non-wandering set can be reduced to the dynamics of a one-dimensional model - namely, those of the quadratic family, given by $p_c(x)=x^2+d$, $d\in[-2,\frac{1}{4}]$ - thus giving an approximate geometric model for the flow.\\

To introduce our results in greater detail, let us first recall that whenever we have $C^2-4AB>0$, the Rössler system (as given by Eq.\ref{Vect}) generates two fixed points, both saddle-foci (of opposing indices), $P_{In}=(\frac{C-\sqrt{C^2-4AB}}{2},-\frac{C-\sqrt{C^2-4AB}}{2A},\frac{C-\sqrt{C^2-4AB}}{2A})$ and $P_{Out}=(\frac{C+\sqrt{C^2-4AB}}{2},-\frac{C+\sqrt{C^2-4AB}}{2A},\frac{C+\sqrt{C^2-4AB}}{2A})$. In particular, $P_{In}$ has a two-dimensional unstable manifold $W^u_{In}$, while $P_{Out}$ has a two-dimensional stable manifold $W^s_{Out}$ (see the illustration in Fig.\ref{loci}). Conversely, $P_{In}$ generates a one-dimensional stable manifold $W^s_{In}$, while $P_{Out}$ generates a one-dimensional unstable manifold $W^u_{Out}$. By qualitatively analyzing the vector field, we first prove the following topological criterion for the existence of an attractor (see Th.\ref{attrac} in Sect.2):

\begin{claim}
\label{th1}    Consider any $(A,B,C)$ s.t. $C-4AB>0$. Then, there exists a hyperplane $S$, varying smoothly with $B$, s.t. whenever $W^s_{Out}\cap S$ is a closed curve homotopic to $S^1$ the flow generates an attractor $A$. Moreover, $A$ is robust under sufficiently small $C^1$ perturbations.
\end{claim}
As we will see, the attractor $A$ closely correlates with the results of the numerical studies. Moreover, it is easy to see Th.\ref{th1} is essentially an existence theorem - as such, it does not teach us too much about the dynamics of the said attractor: for example, it is easy to see it does not teach us if $A$ is a chaotic attractor or not. And indeed, by the numerical evidence there is no reason to assume the attractor $A$ is necessarily chaotic - as observed in many numerical studies, there are parameter values at which the Rössler attractor is a stable, attracting periodic trajectory (see, for example, \cite{G}, \cite{MBKPS} or \cite{BBS}). Therefore, in order to describe the dynamics of the Rössler system on its non-wandering set we take a different approach - which we do in Th.\ref{th3}.\\

In order to introduce Th.\ref{th3}, let us first recall the notion of a \textbf{trefoil parameter} (see Def.\ref{def32} in the next section). Roughly speaking, a parameter $(A,B,C)$ is a trefoil parameter provided the corresponding Rössler system generates a heteroclinic trefoil knot in the $3-$sphere $S^3$. Trefoil parameters were originally introduced by the author in \cite{I} as an idealized form of the Rössler system - and as proven in \cite{I}, at trefoil parameters the dynamics of the Rössler are essentially those of a suspended Smale Horseshoe and include infinitely many periodic trajectories (for more details see Th.3.15 in \cite{I}, or Th.\ref{th31} in the next section). Now, let $P$ denote the parameter space of Eq.\ref{Vect}, and recall that for $d\in[-2,\frac{1}{2}]$ we denote by $p_d$ the polynomial $x^2+d$. Using the properties of trefoil parameters, we prove the following (see Th.\ref{polytheorem}):

\begin{claim}
\label{th3}   Let $p\in P$ be a trefoil parameter. Then, there exists a function $\Pi:P\to[-2,\frac{1}{4}]$ s.t. the following is satisfied:
\begin{itemize}
    \item Set $d=\Pi(v)$. Then, there exists a cross-section $U_v$ and a non-empty subset $I_v\subseteq U_v$ on which the first-return map for the flow $f_v:I_v\to I_v$ is well-defined. Moreover, there exists a map $\pi_v:I_v\to\mathbf{R}$ s.t. $\pi_v\circ f_v=p_d\circ\pi_v$.
    \item Given any $n>0$, provided $v$ is sufficiently close to the trefoil parameter $p$, $I_v$ includes at least $n$ distinct periodic orbits for $f_v$, denoted by $\Omega_1,...,\Omega_n$. Moreover, $\pi_v$ is continuous on $\Omega_1,...,\Omega_n$ - and for every $1\leq i\leq n$, $\pi_v(\Omega_i)=P_i$ is a periodic orbit for $p_d$ of the same minimal period.
    \item $\Pi$ is continuous at trefoil parameters $p$, and satisfies $\Pi(p)=-2$. Moreover, $\pi_p(I_p)$ includes all the periodic orbits of $p_{-2}$ on the real line.
\end{itemize}
\end{claim}
Th.\ref{th3} has the following meaning - given a parameter $v$ sufficiently close to a trefoil parameter $p$, the flow dynamics around its periodic trajectories are essentially those of a suspended quadratic polynomial (for illustrations, see Fig.\ref{match1} and Fig.\ref{template}). Let us recall the dynamics of the Rössler attractor are long known from numerical studies to behave similarly to the quadratic family (see, for example, \cite{MBKPS}, \cite{GKP} and \cite{BBS})- therefore, Th.\ref{th3} can be considered as an analytic counterpart of these numerical observations. Finally, let us remark that even though it is not at all obvious from the text, the results of this paper were strongly influenced by notions and ideas originating in \cite{PY} , \cite{ST}, \cite{BeH}, \cite{Y}, and \cite{Pi}. In particular, Th.\ref{th3} originated by an attempt to prove the existence of a period-doubling cascade for the Rössler system.
\section{Preliminaries}
In this section we discuss and introduce many facts and notions on which the results of this paper are based. This section is organized as follows - we begin with a survey of the basics of the Rössler system and its dynamics (see Section $1.1$). Following that, in Section $1.2$ we go over several basic facts about the dynamics of real, quadratic polynomials.
\subsection{Chaotic dynamics in the Rössler system.}

From now on given $(a,b,c)\in\mathbf{R}^3$, we switch to the more convenient form of the Rössler system:
\begin{equation} \label{Field}
\begin{cases}
\dot{x} = -y-z \\
 \dot{y} = x+ay\\
 \dot{z}=bx+z(x-c)
\end{cases}
\end{equation}
We always denote this vector field corresponding to $(a,b,c)\in\mathbf{R}^3$ by $F_{a,b,c}$. This definition is slightly different from the one presented in Eq.\ref{Vect} - however, setting $p_1=\frac{-C+\sqrt{C^2-4AB}}{2A}$, it is easy to see that whenever $C^2-4AB>0$, $(X,Y,Z)=(x-ap_1,y+p_1,z-p_1)$ defines a change of coordinates between the vector fields in Eq.\ref{Vect} and Eq.\ref{Field}. Since the vector field in Eq.\ref{Field} depends on three parameters, $(a,b,c)$, we now specify the region in the parameter space in which we prove our results. The parameter space $P\subseteq\mathbf{R}^3$ we consider throughout this paper is composed of parameters satisfying the following\label{eq:9}:

\begin{itemize}
\item \textbf{Assumption $1$} -  for every parameter $p\in P,p=(a,b,c)$ the parameters satisfy $a,b\in(0,1)$ and $c>1$. For every choice of such $p$, the vector field $F_p$ given by Eq.\ref{Field} always generates precisely two fixed points - $P_{In}=(0,0,0)$ and $P_{Out}=(c-ab,b-\frac{c}{a},\frac{c}{a}-b)$.
\item \textbf{Assumption 2 }- for every $p\in P$ the fixed points $P_{In},P_{Out}$ are both saddle-foci of opposing indices. In more detail, we always assume that $P_{In}$ has a one-dimensional stable manifold, $W^s_{In}$, and a two-dimensional unstable manifold, $W^u_{In}$. Conversely, we always assume $P_{Out}$ has a one-dimensional unstable manifold, $W^u_{Out}$, and a two-dimensional stable manifold, $W^s_{Out}$ (see the illustration in Fig.\ref{loci} and Fig.\ref{crosx}). 
\item \textbf{Assumption 3 }- For every $p\in P$, let $\gamma_{In}<0$ and $\rho_{In}\pm i\psi_{In}$, $\rho_{In}>0$ denote the eigenvalues of $J_p(P_{In})$, the linearization of $F_p$ at $P_{In}$, and set $\nu_{In}=|\frac{\rho_{In}}{\gamma_{In}}|$. Conversely, let $\gamma_{Out}>0$, $\rho_{Out}\pm i\psi_{Out}$ s.t. $\rho_{Out}<0$ denote the eigenvalues of $J_p(P_{Out})$, the linearization at $P_{Out}$, and define $\nu_{Out}=|\frac{\rho_{Out}}{\gamma_{Out}}|$. We will refer to $\nu_{In},\nu_{Out}$ as the respective saddle indices at $P_{In},P_{Out}$, and we will always assume $(\nu_{In}<1)\lor(\nu_{Out}<1)$ - that is, for every $p\in P$ at least one of the fixed points satisfies the \textbf{Shilnikov condition} (see Th.13.8 in \cite{SSTC} or \cite{LeS} for more details on the connection between the Shilnikov condition and the onset of chaos). 
\end{itemize}

It is easy to see the parameter space $P$ we are considering is not only open in $\mathbf{R}^3$ - moreover, it also includes the region considered in many numerical studies (see, for example, \cite{MBKPS}, \cite{BBS} and \cite{GKP} - among others). As proven in \cite{I}, given any parameter $(a,b,c)=v\in P$ we have the following result about the global dynamics of the flow:

\begin{claim}
  \label{th21}  For every parameter $v\in P$, the vector field $F_v$ satisfies the following:
\begin{itemize}
    \item $F_v$ extends to a continuous vector field on $S^3$, with $\infty$ added as a fixed point for the flow (in $S^3$). As such, $F_v$ has precisely three fixed points in $S^3$ - $P_{In},P_{Out}$ and $\infty$.
    \item There exists an unbounded, one-dimensional invariant manifold $\Gamma_{In}\subseteq W^s_{In}$, s.t. $\overline{\Gamma_{In}}$ connects $P_{In},\infty$.
    \item There exists an unbounded, one-dimensional invariant manifold $\Gamma_{Out}\subseteq W^u_{Out}$, s.t. $\overline{\Gamma_{Out}}$ connects $P_{Out},\infty$.
    \item Moreover, $\Gamma_{In},\Gamma_{Out}$ are not knotted with one another (see the illustration in Fig.\ref{fig7}).
\end{itemize}
\end{claim}
For a proof, see Th.2.8 in \cite{I} (for an illustration, see Fig.\ref{fig7}). To continue, we now introduce the following notion:

\begin{definition}
    \label{nowandering} Let $F$ be a $C^1$ vector fields of $\mathbf{R}^3$, and let $\phi_t$, $t\in\mathbf{R}$ denotes the corresponding flow. Then, the \textbf{\textit{non-wandering}} set would be defined as $\{x\in\mathbf{R}^3|\lim_{t\to\infty}\phi_t(x)\ne\infty\}$ - i.e., the collection of initial conditions whose trajectories are not attracted to $\infty$.
\end{definition}
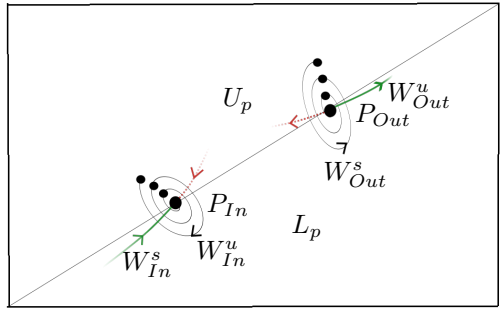
\begin{figure}[h]
\centering
\begin{overpic}[width=0.4\textwidth]{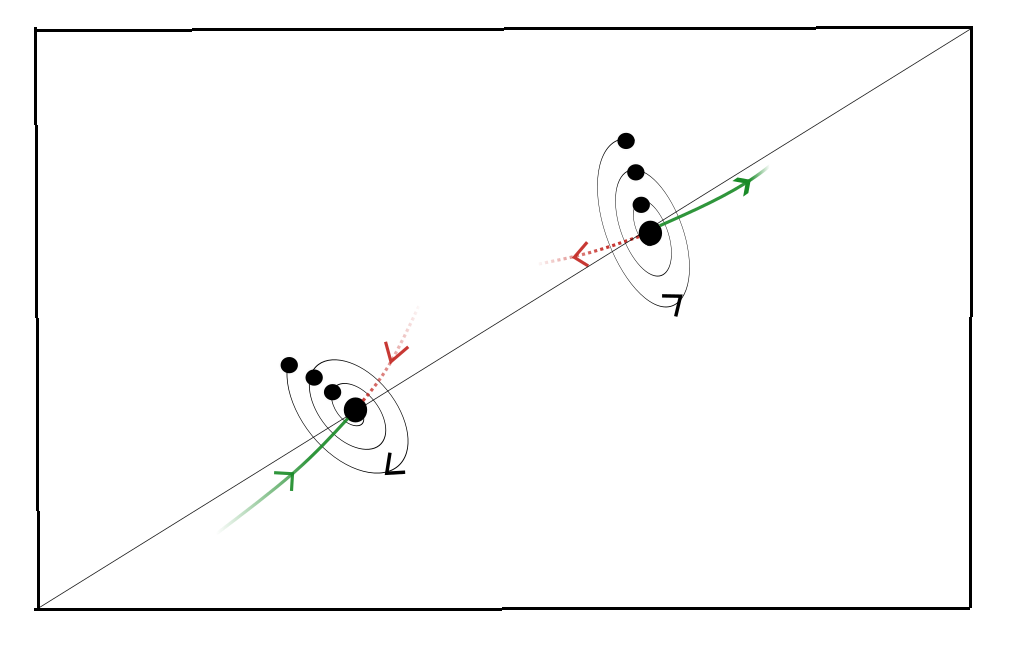}%
\put(630,280){$W^s_{Out}$}
\put(750,430){$W^u_{Out}$}
\put(690,390){$P_{Out}$}
\put(390,130){$W^u_{In}$}
\put(410,220){$P_{In}$}
\put(440,420){$U_p$}
\put(570,180){$L_p$}
\put(250,110){$W^s_{In}$}
\end{overpic}
\caption[Fig2]{The local dynamics around the fixed points on $Y_p$ - $l_p$ is the straight line separating $U_p,L_p$. The green and red flow lines are the one-dimensional separatrices in $W^s_{In},W^u_{Out}$.}\label{loci}
\end{figure}

It is easy to see the non-wandering set w.r.t. to any $C^1$ vector field $F$ includes all the periodic trajectories for the flow. Now, given a parameter $v\in P,v=(a,b,c)$, begin by considering the cross-section $Y=\{\dot{y}=0\}=\{(x,-\frac{x}{a},z)|x,z\in\mathbf{R}\}$ (see Eq.\ref{Field}), and consider its sub-curve $l_v=\{(x,-\frac{x}{a},\frac{x}{a})|x\in\mathbf{R}\}$. Because the normal vector to $Y$ is $N=(1,a,0)$, it follows by direct computation that $l_v=\{s\in Y|F_v(s)\bullet N=0\}$, and $P_{In},P_{Out}\in l_v$. It follows $Y\setminus l_v$ constitutes of two components, both half planes, parameterized as follows:

\begin{itemize}
    \item $U_v=\{(x,-\frac{x}{a},z)|x\in\mathbf{R}, \frac{x}{a}<z\}$ - that is, the upper half plane (see the illustration in Fig.\ref{loci}). .
    \item $L_v=\{(x,-\frac{x}{a},z)|x\in\mathbf{R}, \frac{x}{a}>z\}$ - that is, the lower half plane (see the illustration in Fig.\ref{loci}).
\end{itemize}

By the definition above, the vector field $F_v$ is transverse to both $U_v$ and $L_v$. It is also easy to see both $U_v$ and $L_v$ vary smoothly when the parameter $v$ are varied smoothly in $P$. As proven in Lemma 2.1 and Lemma 2.2 in \cite{I} we have:

\begin{corollary}
\label{TR}    For any $v\in P$, the cross-section $U_v$ defined above satisfies the following:
    \begin{itemize}
        \item The two-dimensional $W^u_{In},W^s_{Out}$ are transverse to $U_v$ at $P_{In},P_{Out}$ (respectively).
        \item Let $T$ be a periodic trajectory for $F_v$. Then, $T$ is transverse to the closed half-plane $\overline{U_v}$ in at least one point. 
    \end{itemize}
\end{corollary}
\begin{remark}
    In fact, one can prove something stronger - namely, that if $x$ is an initial condition in the non-wandering set for $F_v$ which does not lie on the one-dimensional invariant manifold $W^s_{In}$, its trajectory eventually intersects $U_v$ transversely (see Lemma 2.1 in \cite{I}). We, however, will not need it.
\end{remark}
\begin{figure}[h]
\centering
\begin{overpic}[width=0.4\textwidth]{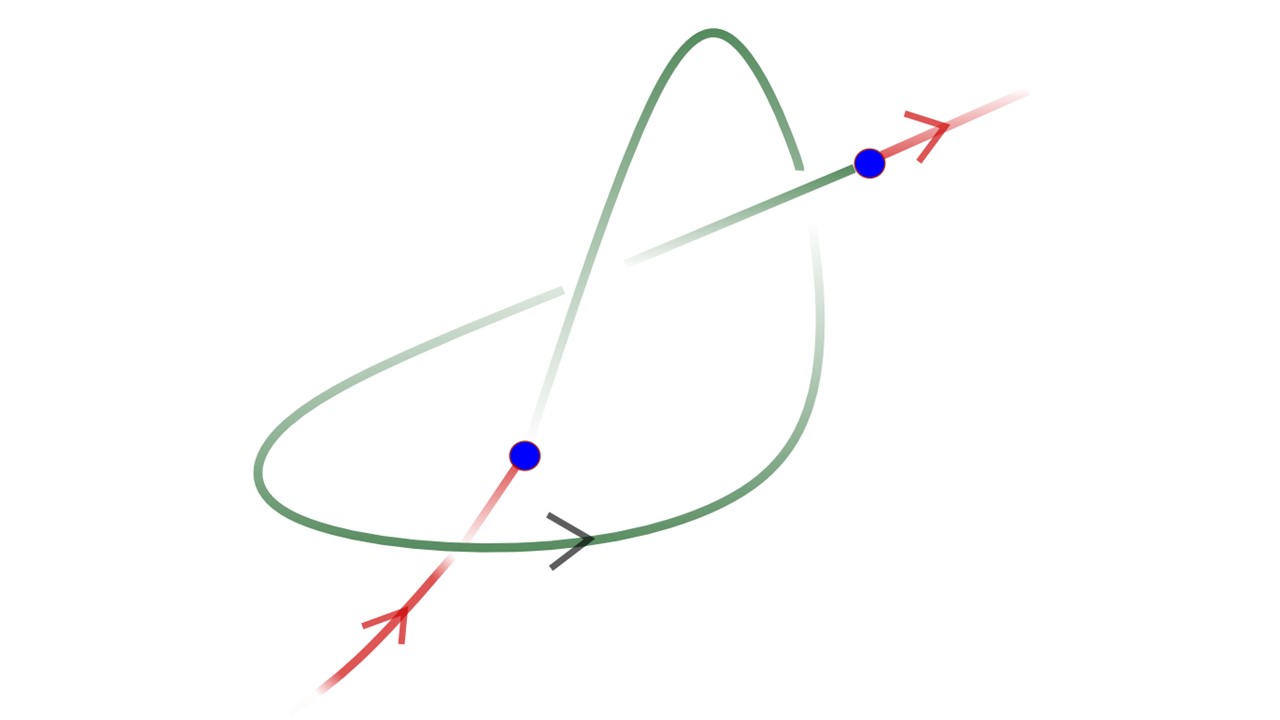}
\put(440,210){$P_{In}$}
\put(340,70){$\Gamma_{In}$}
\put(400,350){$\Theta$}
\put(660,370){$P_{Out}$}
\put(780,430){$\Gamma_{Out}$}
\end{overpic}
\caption[Fig7]{A heteroclinic trefoil knot (see Def.\ref{def32}). $\Theta$ denotes the bounded heteroclinic trajectory, while $\Gamma_{In},\Gamma_{Out}$ denote the unbounded heteroclinic trajectories given by Th.\ref{th21}.}
\label{fig7}
\end{figure}
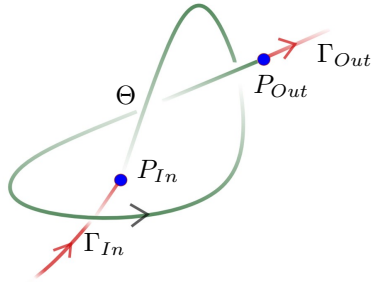
As stated in the introduction, in \cite{I} the author had proven a criterion for the existence of complex dynamics for the Rössler system. Since the proofs of both Th.\ref{attrac} and \ref{polytheorem} are heavily based on that criterion, let us now introduce it. In order to do so, we first make the following observation - assume $p\in P$ is a parameter s.t. the vector field $F_p$ generates a bounded heteroclinic trajectory $\Theta$ which flows from $P_{Out}$ to $P_{In}$ (in particular, $\Theta=W^s_{In}\cap W^u_{Out}$ - see the illustration in Fig.\ref{fig7}). Now, consider the set $\Lambda=W^s_{In}\cup W^u_{Out}\cup\{P_{In},P_{Out},\infty\}$. Since $W^s_{In}=\Theta\cup\Gamma_{In}$ and $W^u_{Out}=\Theta\cup\Gamma_{Out}$ by Th.\ref{th21} it immediately follows $\Lambda$ is a knot in $\mathbf{R}^3$. Motivated by this observation, from now on we consider a very specific type of such heteroclinic knots, trefoil parameters, defined below:

\begin{definition}\label{def32}
With the notations above, we say $p=(a,b,c)\in P$ is a \textbf{trefoil parameter for the Rössler system} provided the following three conditions are satisfied by the vector field $F_p$:
\begin{itemize}
    \item There exists a bounded heteroclinic trajectory $\Theta$ as in Fig.\ref{fig7}. Consequentially, $\Lambda$ (as defined above) forms a trefoil knot in $S^3$.
    \item The two-dimensional manifolds $W^u_{In}$ and $W^s_{Out}$ coincide. This condition implies $\overline{W^u_{In}}=\overline{W^s_{Out}}$ forms the boundary of an open topological ball - which, from now on, we always denote by $B_\alpha$. It is easy to see $\Theta\subseteq B_\alpha$, while $\Gamma_{In},\Gamma_{Out}\not\subseteq B_\alpha$.
    \item $\Theta\cap \overline{U_p}=\{P_0\}$ is a point of transverse intersection - see the illustration at Fig.\ref{intersect}.
\end{itemize}.
\end{definition}

\begin{figure}[h]
\centering
\begin{overpic}[width=0.4\textwidth]{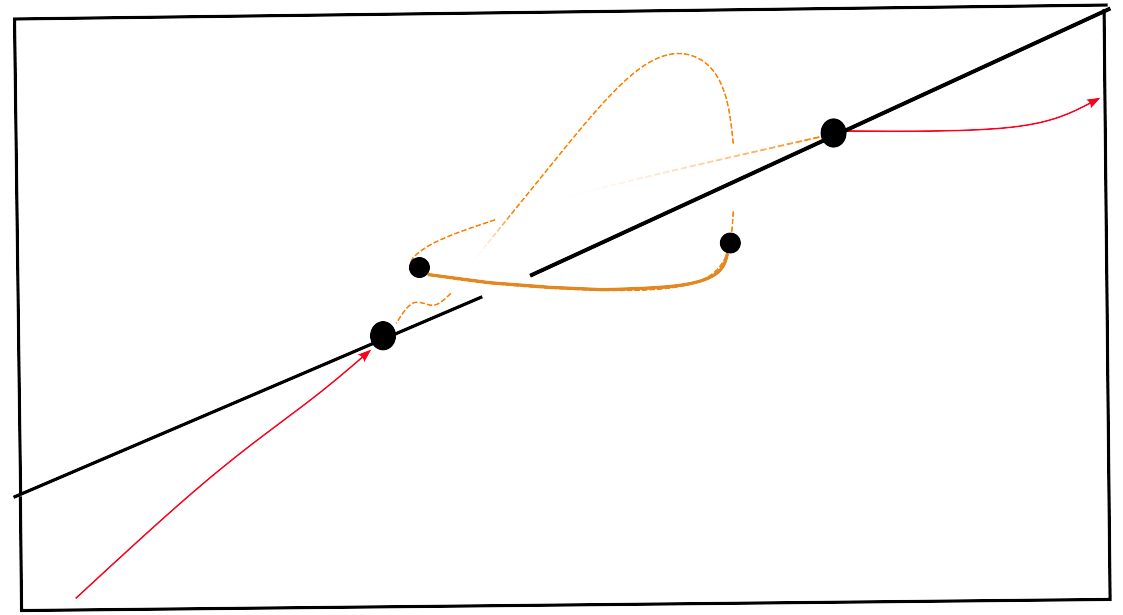}
\put(480,250){$\Theta$}
\put(320,180){$P_{In}$}
\put(160,50){$\Gamma_{In}$}
\put(310,340){$P_0$}
\put(670,310){$P_1$}
\put(100,450){$U_p$}
\put(700,50){$L_p$}
\put(730,370){$P_{Out}$}
\put(900,400){$\Gamma_{Out}$}
\end{overpic}
\caption[The intersection of the heteroclinic trefoil with $\{\dot{y}=0\}$.]{\textit{The heteroclinic trajectory $\Theta$ (for a trefoil parameter) winds once around $P_{In}$ - hence it intersects the half-plane $U_p$ at $P_0$ and the half-plane $L_p$ at $P_1$.}}
\label{intersect}
\end{figure}
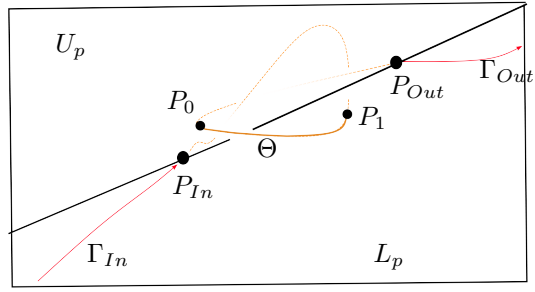

Trefoil parameters were originally introduced in \cite{I} as an idealized form of the Rössler system (see the discussion at the beginning of Section $3$ in \cite{I}) - in particular, the existence of parameters in $P$ at which the Rössler system generates a heteroclinic trefoil knot was observed numerically (see Fig.$5.B1$ in \cite{MBKPS}). For more details on the dynamics of the Rössler system at trefoil parameters, see \cite{I} and \cite{I2}. The reason we are interested in trefoil parameters is because at trefoil parameters the Rössler system is chaotic. Namely, as proven at Th.3.15 and Cor.3.26 in \cite{I}, we have the following:

\begin{theorem}\label{th31}
 Let $p\in P$ be a trefoil parameter for the Rössler system, and let  $f_p:\overline{U_p}\to \overline{U_p}$ denote the corresponding first-return map (wherever defined). Additionally, let $\sigma:\{1,2\}^\mathbf{N}\to\{1,2\}^\mathbf{N}$ denote the one-sided shift - then, we have the following:
 \begin{itemize}
     \item There exists a curve $\rho\subseteq\overline{U_p}$ s.t. $\overline{U_p}\setminus\rho$ consists of two components, $D_1,D_2$ (see the illustration in Fig.\ref{part}).
     \item $P_{In}\in D_1$ (see the illustration in Fig.\ref{part}).
     \item There exists an $f_p$-invariant set $Q\subseteq\overline{U_p}\setminus\rho$ s.t. $Q$ is a subset of the non-wandering set for the flow.
     \item There exists a factor map $\pi:Q\to\{1,2\}^\mathbf{N}$ s.t. $\pi\circ f_p=\sigma\circ\pi$. Additionally, both $f_p$ and $\pi$ are continuous on $Q$.
     \item If $s\in\{1,2\}^\mathbf{N}$ is periodic of minimal period $k>0$, then $\pi^{-1}(s)$ includes at least one periodic point $x_s$ for $f_p$. Moreover, the minimal period of $x_s$ w.r.t. $f_p$ is $k$. 
 \end{itemize}

In particular, at trefoil parameters the Rössler system has infinitely many periodic trajectories.
\end{theorem}

See the illustration in Fig.\ref{part}. The idea behind the proof of Th.\ref{th31} is that at trefoil parameters the topology of the heteroclinic trefoil knot forces the first-return map to behave essentially like a Smale Horseshoe (see the illustration in Fig.\ref{fig16}) - for more details, see Th.3.25 and Th.3.15 in \cite{I}. Consequentially, at trefoil parameters one can describe the dynamics of the first-return map on initial conditions in $\overline{U_p}\setminus\rho$ by using a symbolic coding (w.r.t. $D_1,D_2$ mentioned above).\\

\begin{figure}[h]
\centering
\begin{overpic}[width=0.4\textwidth]{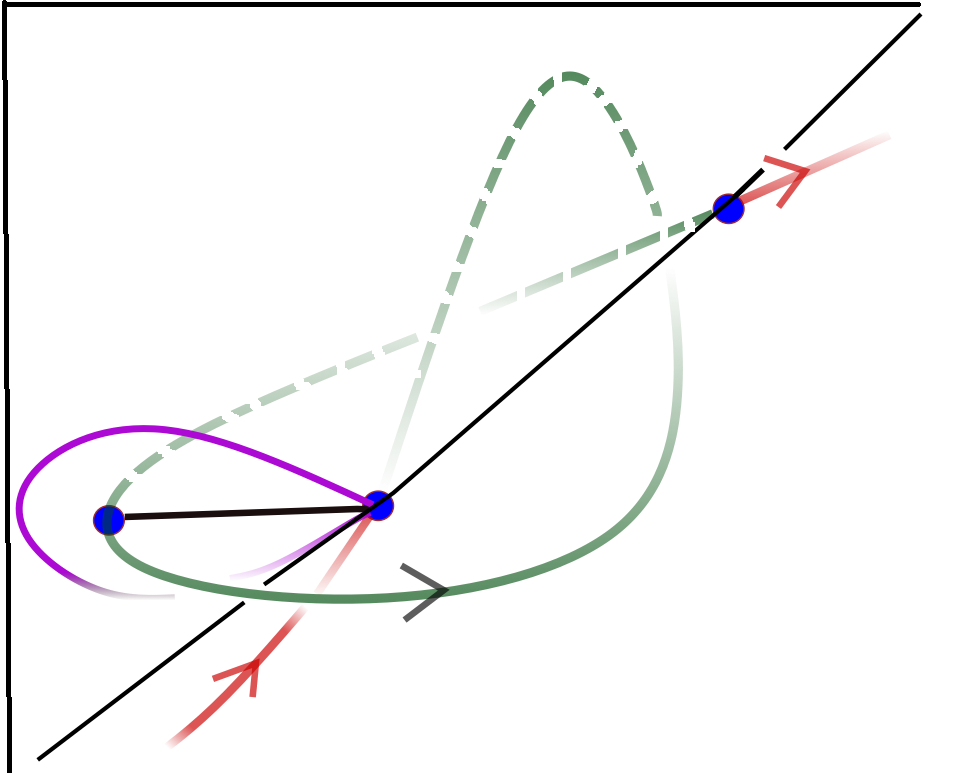}
\put(420,250){$P_{In}$}
\put(250,290){$\gamma$}
\put(60,300){$P_0$}
\put(100,390){$f_p(\gamma)$}
\put(800,550){$P_{Out}$}
\put(560,470){$\Theta$}
\end{overpic}
\caption[Suspending a curve along the trefoil knot.]{\textit{Flowing $\gamma$ along the trefoil. $\Theta$ is the green trajectory. Since $f_p(\gamma)$ is a closed loop, we expect the first-return map for the flow to behave similarly to a Smale Horseshoe.}}
\label{fig16}
\end{figure}
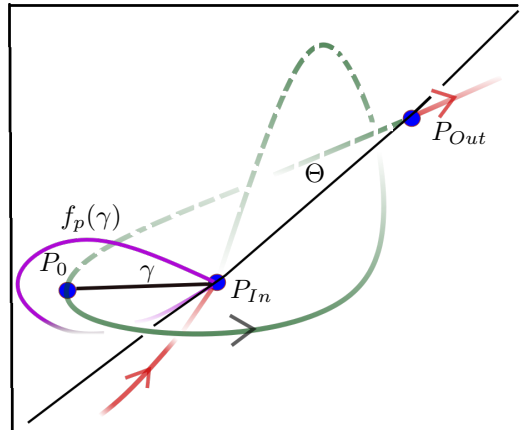

To continue, as stated at the introduction, in this paper we will match the flow dynamics with one-dimensional polynomial dynamics by studying their shared symbolic dynamics - and it is Th.\ref{th31} which allows us to do so. To begin, given a trefoil parameter $p\in P$, recall that by Cor.\ref{TR} when we deform the vector field $F_p$ to $F_v$ through the parameter space (for some $v\in P$) the cross-section $\overline{U_p}$ varies continuously to $\overline{U_v}$ (both are closed half-planes). As such, it follows the curve $\rho\subseteq\overline{U_p}$ is continuously deformed to some $\rho_v\subseteq \overline{U_v}$. Consequentially, since $\rho$ divides $\overline{U_p}$ to $D_1,D_2$, $\rho_v$ divides $\overline{U_v}$ to $D_{1,v},D_{2,v}$ (see the illustration in Fig.\ref{part}). Since by Th.\ref{th31} we have $P_{In}\in D_1$, it is easy to see $P_{In}\in D_{1,v}$.\\

Now, let $f_v:\overline{U_v}\to\overline{U_v}$ denote the first-return map for $F_v$ (wherever defined), and let $I_v$ denote the collection of initial conditions $x\in\overline{U_v}\setminus\rho_v$ satisfying the following:

\begin{itemize}
    \item $x$ is in the non-wandering set of $F_v$ - i.e., the trajectory of $x$ is not attracted to $\infty$ (see Def.\ref{nowandering}). 
    \item For every $k>0$, $f^k_v(x)$ is defined and lies in $\overline{U_v}\setminus\rho_v$.
\end{itemize}

By the discussion above it follows we can again define a symbolic coding  on $I_v$ - that is, there exists a function $\pi_v:I_v\to\{1,2\}^\mathbf{N}$ (not necessarily continuous) s.t. $\pi_v\circ f_v=\sigma\circ\pi_v$ (again, with $\sigma:\{1,2\}^\mathbf{N}\to\{1,2\}^\mathbf{N}$ denoting the one-sided shift). Moreover, let us note that since for all $v\in P$, $P_{In}\in I_v$, by $f_v(P_{In})=P_{In}$ (as $P_{In}$ is a fixed-point) it follows the constant $\{1,1,1...\}$ is in $I_v$ - i.e., for all $v\in P$, $I_v$ is never empty. We therefore summarize our results as follows:
\begin{corollary}
    \label{invariant} For any $v\in P$, the set $I_v$ is non-empty, and includes the fixed point $P_{In}$ - hence, $\{1,1,1,...\}\in\pi_v(I_v)$. Moreover, $I_v$ is a subset of the non-wandering set of the vector field $F_v$.
\end{corollary}

However, one can say more. As proven in Th.4.1 in \cite{I}, w.r.t. this coding we have the following result, with which we conclude this subsection:
 \begin{claim}\label{th41}
     Let $p\in P$ be a trefoil parameter, and let $s\in\{1,2\}^\mathbf{N}$ be periodic of minimal period $k$. Then, provided the parameter $v$ is sufficiently close to $p$, $\pi_v^{-1}(s)$ includes a periodic point $x_s$ for $f_v$, of minimal period $k$. Moreover, both $f_v,\pi_v$ are continuous at $x_s,f_v(x_s),...,f^{k-1}_v(x_s)$. Consequentially, for every $n>0$ there exists some $\epsilon>0$ s.t. whenever $||v-p||<\epsilon$, the Rössler system corresponding to $v$ generates at least $n$ distinct periodic trajectories.
 \end{claim}

 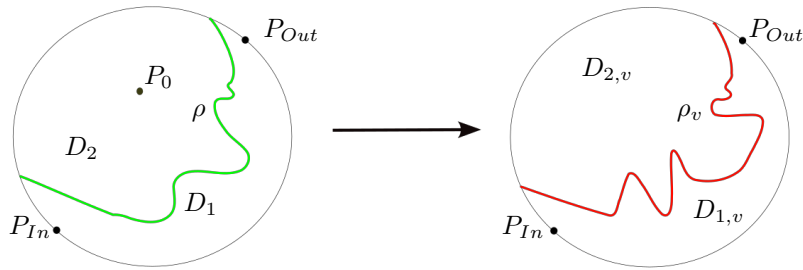
\begin{figure}[h]
\centering
\begin{overpic}[width=0.6\textwidth]{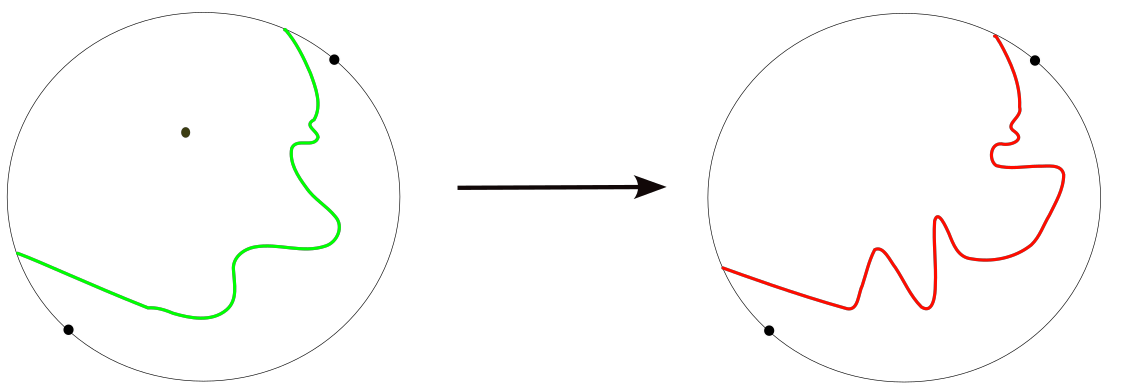}
\put(0,50){$P_{In}$}
\put(620,50){$P_{In}$}
\put(170,240){$P_0$}
\put(230,200){$\rho$}
\put(840,200){$\rho_v$}
\put(70,150){$D_2$}
\put(720,250){$D_{2,v}$}
\put(220,80){$D_1$}
\put(860,70){$D_{1,v}$}
\put(320,300){$P_{Out}$}
\put(930,300){$P_{Out}$}
\end{overpic}
\caption[Fig32]{the deformation of $U_p$ where $p$ is a trefoil parameter (on the left) to $U_v$ (on the right) - for simplicity, $U_p,U_v$ are drawn as discs rather than half-planes. As can be seen, the curve $\rho$ is deformed to $\rho_v$.}
\label{part}
\end{figure}
In other words, Th.\ref{th41} states that the closer a parameter $v$ is to a trefoil parameter, the more complex are the dynamics of the Rössler system corresponding to $v$ - i.e., as $v\to p$ the number of periodic trajectories for the flow increases. At Section 3 we will apply both Th.\ref{th31} and Th.\ref{th41} to reduce the dynamics of the Rössler system to those of quadratic polynomials on the real line - in order to do so we will also need several facts from the theory of one-dimensional dynamics, which brings us to the following subsection:

\subsection{Symbolic dynamics for the quadratic family.}
Recall that given any quadratic polynomial $p(x)=ax^2+bx+c$, $a,b,c\in\mathbf{R}$, $p$ can be conjugated to a normal form of the type $p_c(x)=x^2+c$ (for some $c\in\mathbf{R}$). Further recall that when $c\not\in[-2,\frac{1}{4}]$, for a generic $x\in\mathbf{R}$, we have $p^n_c(x)\to\infty$. However, when $c\in[-2,\frac{1}{4}]$, there is always a non-trivial closed, bounded interval $V_c\subseteq\mathbf{R}$, s.t. $p_c(V_c)\subseteq V_c$ - and for every $x\not\in V_c$, $p^n_c(x)\to\infty$. Moreover, writing $V_c=[x_{2,c},x_{1,c}]$ we always have $x_{2,c}<0<x_{1,c}$, i.e., the sequence $\{p^n_c(0)\}_n$ is bounded - that is, $p_c$ folds $V_c$ on itself in some way (see the illustration in Fig.\ref{logistic}). For a proof of these results, see, for example, Ch.1.10 in \cite{GN}.\\

With these ideas in mind we define the \textbf{Quadratic Family} by $p_c(x)=x^2-c,c\in[-2,\frac{1}{4}]$. Now, let $x_0\in V_c$ be some periodic point for $p_c$ of minimal period $k$ (for some $c\in[-2,\frac{1}{2}]$). Following the terminology of \cite{DeMVS} and \cite{CG}, we classify its type as follows:
\begin{itemize}
    \item When $0<|\frac{d p^k(x_0)}{dx}|<1$, we say $x_0$ is a \textbf{hyperbolic} or \textbf{attracting} (we use these two terms interchangeably) - in that case, there exists some open interval $V_0\subseteq V_c$ s.t. for $x\in V_0$, $\lim_{n\to\infty}p^{kn}_c(x)=x_0$. 
    \item $x_0$ is  \textbf{super-attracting} if $\frac{dp^k(x_0)}{dx}=0$ - similarly, there exists an open interval $V_0\subseteq  V_c$ s.t. for $x\in V_0$, $\lim_{n\to\infty}p^{kn}_c(x)=x_0$.
    \item When $|\frac{dp^k_c(x_0)}{dx}|=1$, we say $x_0$ is \textbf{weakly attracting} (or \textbf{parabolic}). In that case, there exists an open interval $V_0\subseteq V_c$, $x_0\in\partial V_0$ s.t. $\lim_{n\to\infty}p^{kn}_c(x)=x_0$. 
    \item Finally, we say $x_0$ is \textbf{repelling} if $|\frac{dp^k_c(x_0)}{dx}|>1$. By Th.2.2 in \cite{CG}, for every $c\in[-2,\frac{1}{4}]$ $p_c$ has at most one periodic point $x_0\in V_c$ which is not repelling.
\end{itemize}

\begin{figure}[h]
\centering
\begin{overpic}[width=0.3\textwidth]{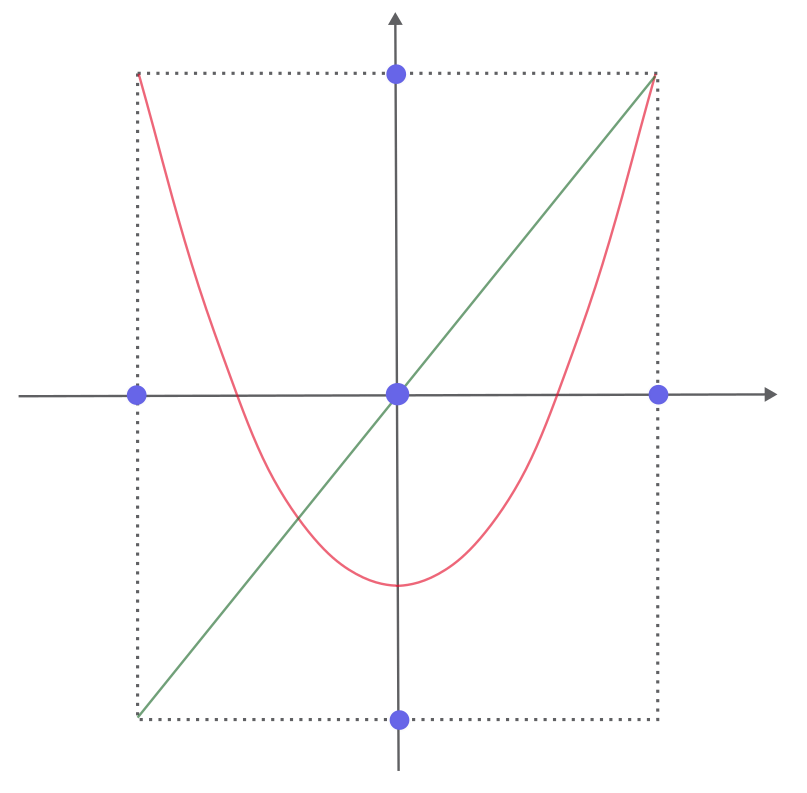}
\put(500,750){$\{y=x\}$}
\put(580,240){$p_c(I_c)$}
\put(70,435){$x_{2,c}$}
\put(820,435){$x_{1,c}$}
\put(475,990){$y$}
\put(970,485){$x$}
\put(500,50){$x_{2,c}$}
\put(500,435){$0$}
\put(500,910){$x_{1,c}$}

\end{overpic}
\caption[Fig31]{The invariant interval $I_c=[x_{2,c},x_{1,c}]$, for some $-2<c<0$. $p_c$ has one fixed point in $[x_{2,c},0)$ and another at $x_{1,c}$. }
\label{logistic}
\end{figure}
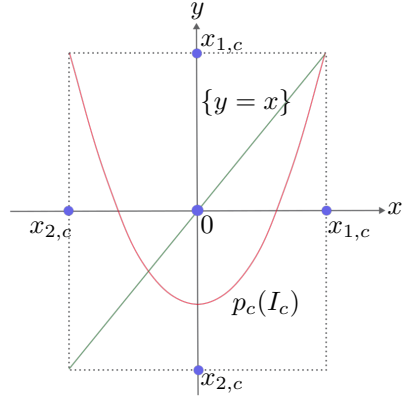

Now, recall $\sigma:\{1,2\}^\mathbf{N}\to\{1,2\}^\mathbf{N}$ denotes the one-sided shift, and that given any $c\in[-2,\frac{1}{4}]$ the critical point $0$ is strictly interior to $V_c$. Therefore, we can consider the sub-intervals $(0,x_{1,c}]=I_1$, $[x_{2,c},0)=I_2$ and define a symbolic coding for initial conditions in the maximal invariant set of $p_c$ in $I_c\setminus\{0\}$, denoted by $I_c$ - that is, there exists a continuous map $\xi_c:I_c\to\{1,2\}^\mathbf{N}$ s.t. $\xi_c\circ p_c=\sigma\circ\xi_c$, s.t. $\xi_c(x)=\{i_n(p^n_c(x))\}_{n\geq0}$, where $i_n(y)=1$ when $y\in I_1$, and $2$ where $y\in I_2$. For example, with previous notations, for every $c\in[-2,\frac{1}{4}]$, $\xi_c(x_{1,c})$ is the constant $\{1,1,1...\}$ - while $\xi_c(x_{2,c})$ is the pre-periodic symbol $\{2,1,1,...\}$.
Following Ch.II.3 in \cite{DeMVS}, we extend $\xi_c$ to $0$ as follows - for $c\in[-2,\frac{1}{4}]$, define the \textbf{Kneading Invariant} of $p_c$ by $K(c)=\{i^+_n(p^n_c(0))\}_n$, where $i^+_n(x)=\lim_{y\uparrow x}i_n(y)$. See the illustration in Fig.\ref{partition}.\\

\begin{figure}[h]
\centering
\begin{overpic}[width=0.8\textwidth]{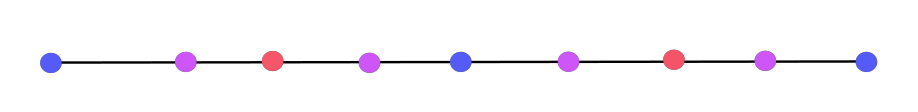}
\put(170,0){$p^{-2}_c(0)$}
\put(270,0){$p^{-1}_c(0)$}
\put(380,0){$p^{-2}_c(0)$}
\put(590,0){$p^{-2}_c(0)$}
\put(710,0){$p^{-1}_{c}(0)$}
\put(810,0){$p^{-2}_{c}(0)$}
\put(495,0){$0$}
\put(935,0){$x_{1,c}$}
\put(20,0){$x_{2,c}$}
\end{overpic}
\caption[Fig31]{The invariant interval $I_c=[x_{2,c},x_{1,c}]$, divided to different components by $\cup_{n\geq0}p^{-n}_c(0)$. }
\label{partition}
\end{figure}
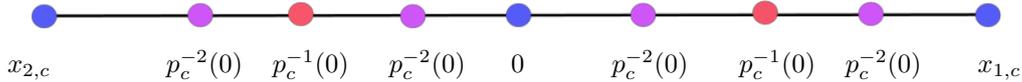

Given any $c\in[-2,\frac{1}{4}]$, the kneading invariant determines which symbols can (and cannot) appear in $\xi_c(I_c)$. In fact, using Kneading Theory one can prove that given $c\in[-2,\frac{1}{4}]$ s.t. $\xi_c(I_c)$ includes every periodic $s\in\{1,2\}^\mathbf{N}$, then $c=-2$ - see Th.\ref{maximal} below. We conclude this section the following result, which will be useful in the proof of Th.\ref{polytheorem} (for a proof, see Th.II.3.2 and Cor.II.10.1 in \cite{DeMVS}):
\begin{claim}
    \label{maximal}
   With previous notations, given $c\in[-2,\frac{1}{4}]$, provided $c>-2$, there exists some periodic $s\in\{1,2\}^\mathbf{N}$ s.t. $s\not\in\xi_c(I_c)$. Moreover, given $-2\leq d\leq c$ we have $\xi_c(I_c)\subseteq \xi_d(I_d)$ - that is, the dynamical complexity can only increase as $c\downarrow -2$. In particular, $p_{-2}$ is dynamically maximal, i.e., $\xi_{-2}(I_{-2})$ includes every periodic $s\in\{1,2\}^\mathbf{N}$.
\end{claim}
\begin{remark}
    \label{notemp}
    Given any $c\in [-2,\frac{1}{4}]$, it is easy to see the constant $\{1,1,1...\}$ is always in $I_c$. 
\end{remark}
Th.\ref{maximal} essentially states that given $c\in[-2,\frac{1}{4}]$, $p_c$ covers $V_c$ twice precisely when $c=-2$ - as such, given $c\in[-2,\frac{1}{4}]$, Th.\ref{maximal} proves that the dynamics of $p_{c}$ are essentially a singular Smale Horseshoe precisely when $c=-2$ (we will return to this idea in Section 3 - see Cor.\ref{misiu}). Capitalizing on this idea, let $\kappa\subseteq\{1,2\}^\mathbf{N}$ denote the set of symbols which are not strictly pre-periodic to the constant $\{1,1,1,...\}$. By computation, the kneading invariant for $c=-2$ is the pre-periodic symbol $\{2,1,1,...\}$ (see the illustration in Fig.\ref{p-2}). Now, recall we denote by $\sigma:\{1,2\}^\mathbf{N}\to\{1,2\}^\mathbf{N}$ the one-sided shift - then, the dynamics of $p_{-2}$ on $I_{-2}$ satisfy the following (see Ch.2.6 in \cite{GL}):

\begin{lemma}
    \label{polyhorse}
    The invariant interval $V_{-2}$ for $p_{-2}$ is $[-2,2]$, and $p_{-2}$ is a (branched) double cover of $[-2,2]$ on itself (see the illustration in Fig.\ref{p-2}). Moreover, $\xi_{-2}:I_{-2}\to\kappa$ is a homeomorphism, satisfying $\xi_{-2}\circ p_{-2}\circ\xi^{-1}_{-2}=\sigma$ - consequentially, every component of $I_{-2}$ is a singleton.
\end{lemma}
\begin{figure}[h]
\centering
\begin{overpic}[width=0.3\textwidth]{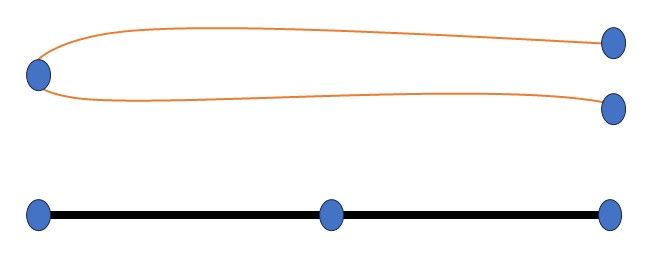}
\put(10,-30){$-2$}
\put(900,-30){$2$}
\put(-170,250){$p_{-2}(0)$}
\put(950,310){$p_{-2}(-2)$}

\put(950,210){$p_{-2}(2)$}

\put(485,-30){$0$}

\end{overpic}
\caption[Fig31]{$p_{-2}$ acts on $[-2,2]$ by folding it twice on itself.}
\label{p-2}
\end{figure}
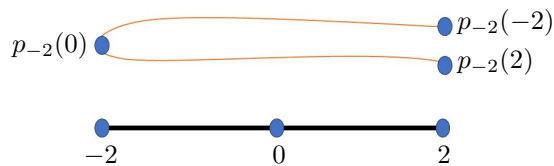
As we will prove in Section $3$, the dynamics of the Rössler system at trefoil parameters are essentially those of $p_{-2}$ on $I_{-2}$ (see Cor.\ref{misiu}) - however, in order to generalize this result to every parameter $v\in P$, not necessarily a trefoil parameter, we will also need the notion of the real Fatou and Julia sets. Following Ch.VI in \cite{DeMVS}, we define these sets as follows:

\begin{definition}
\label{julia}    Let $g:I\to I$ be a continuous interval map. Then, the \textbf{Fatou set} of $g$, $F(g)$, is defined as the set of points around which the sequence $\{g^n\}_n$ is compact (in particular, every component of $F(g)$ is open in $I$). Conversely, we define $J(g)$, the \textbf{Julia set} of $g$, to be $I\setminus F(g)$. 
\end{definition}
Now, let us consider how these ideas manifest for polynomials $p_c(x)=x^2+c$. By Lemma VI.1 in \cite{DeMVS} we know that the real Julia set of $p_c$ at the invariant interval $V_c$, $J(p_c)$, is simply the $\alpha-$limit set of $0$. Additionally, let us recall that by Th.III.2.2 and Th.III.2.3 in \cite{CG} we have the following result:

\begin{claim}
\label{classification} For $c\in[-2,\frac{1}{4}]$, $p_c$ has at most one periodic orbit in $V_d$ that is either attracting or weakly attracting. Consequentially, given $O_1,...,O_n\subseteq V_c$, distinct periodic orbits for $p_c$, $O_1,...,O_{n-1}$ are repelling and lie at the Julia set.
\end{claim}
We conclude this Section with the following fact, which is a corollary of Th.\ref{classification} and previous discussion about the symbolic dynamics of the quadratic family:
\begin{corollary}
\label{single}    For $c\in[-2,\frac{1}{4}]$, let $J(p_c)$ denote the Julia set of $p_c$, and let $F(p_c)$ denote the Fatou set. Then, given a repelling periodic point $x\in J(p_c)$ s.t. $x$ does not lie on the boundary of a component in $F(p_d)$, setting $s=\xi_c(x)$, we have $\xi^{-1}_c(s)=\{x\}$. 
\end{corollary}
\begin{proof}
    Let us write $I_s=\xi^{-1}_c(s)$. It is easy to see $I_s$ is generated by the infinite intersection of nested intervals - hence, by Caratheodory's Theorem on planar convex sets, $I_s$ is the convex hull of either one or two points in $V_c$, i.e., $I_s$ is either an interval or a singleton (for a proof see, for example, \cite{VS}). It would therefore suffice to prove $I_s$ is a singleton - to do so, let us first note that since $J(p_c)$ is the $\alpha$-limit set of $0$, if $I_s$ is not a singleton but an interval, its interior must be a component of $F(p_c)$. Consequentially, if $I_s$ is an interval it immediately follows $x\in\partial I_s$, which implies $x$ lies on the boundary of some component in $F(p_c)$. Since this is not the case per our assumption on $x$, $I_s$ must be a singleton and Cor.\ref{single} now follows.
\end{proof}
\section{Sufficient conditions for the the existence of the Rössler attractor.}
Let $p\in P, p=(a,b,c)$ be a parameter value for the Rössler system, and recall we always denote by $F_p$ the corresponding vector field (see Eq.\ref{Field}). As stated earlier, in this section we prove a topological criterion for the existence of an attractor for the Rössler system (see Th.\ref{attrac} and Prop.\ref{attrac1}). To begin, recall that given parameter values $p=(a,b,c)\in P$, we always assume the parameters $a,b,c$ satisfy $a,b\in(0,1),c>1$. Moreover, we assume the corresponding vector field $F_p$ always generates precisely two fixed points in $\mathbf{R}^3$ - $P_{In}=(0,0,0)$  and $P_{Out}=(c-ab,\frac{ab-c}{a},\frac{c-ab}{a})$, both saddle foci, satisfying the following:
\begin{itemize}
    \item $P_{Out}$ has a stable, two-dimensional invariant manifold $W^s_{Out}$. Conversely, $P_{In}$ has a two-dimensional unstable, invariant manifold $W^u_{In}$ (see the illustration in Fig.\ref{crosx}).
    \item $P_{Out}$ has a one-dimensional unstable invariant manifold $W^u_{Out}$ which consists of two separatrices - $\Gamma_{Out}$, connecting $P_{Out},\infty$ (see Th.\ref{th21}) and $\Delta_{Out}$ ($\Delta_{Out}$ may, or may not, be bounded). Conversely, $P_{In}$ has a one-dimensional stable manifold $W^s_{In}$, composed of two separatrices $\Gamma_{Out}$, connecting $P_{In},\infty$ (again, see Th.\ref{th21}) and $\Delta_{In}$ (which, again, may, or may not, be bounded). See the illustration in Fig.\ref{crosx}.
\end{itemize}

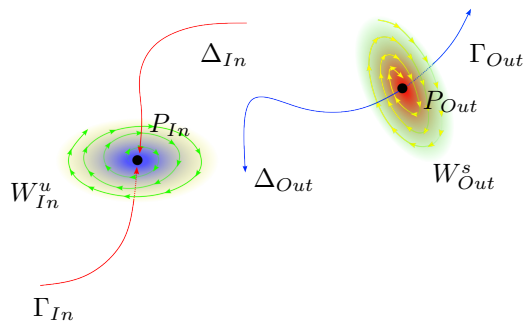
\begin{figure}[h]
\centering
\begin{overpic}[width=0.35\textwidth]{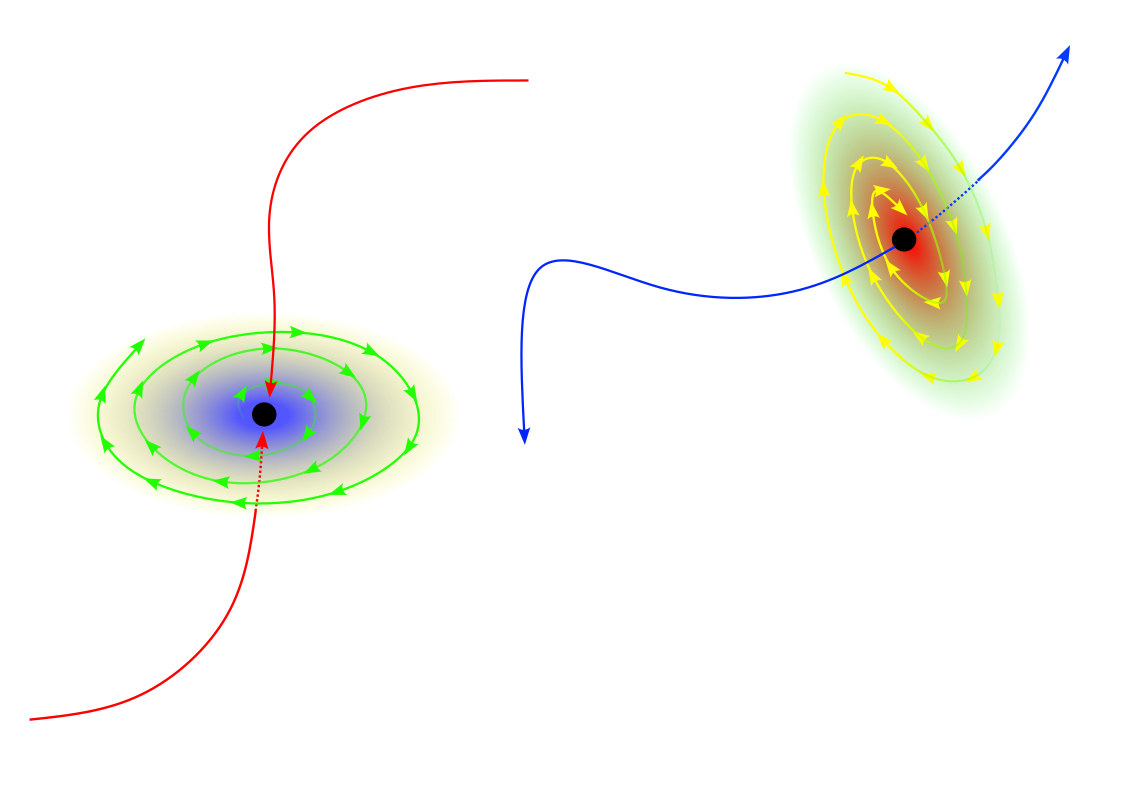}
\put(10,0){$\Gamma_{In}$}
\put(370,550){$\Delta_{In}$}
\put(-40,250){$W^u_{In}$}
\put(950,550){$\Gamma_{Out}$}
\put(850,450){$P_{Out}$}
\put(870,290){$W^s_{Out}$}
\put(260,400){$P_{In}$}
\put(485,280){$\Delta_{Out}$}

\end{overpic}
\caption[Fig31]{The local dynamics around the saddle foci $P_{In}$ and $P_{Out}$. $\Gamma_{In}$ arrives at $P_{In}$ from $\infty$, and $\Gamma_{Out}$ tends to $\infty$. Conversely, each of $\Delta_{In}$ and $\Delta_{Out}$ may (or may not) be bounded.}
\label{crosx}
\end{figure}

Now, given any parameter value $(a,b,c)$, let us consider the hyperplane $S_b=\{(x,y,-b)|x,y\in\mathbf{R}\}$ (see the illustration in Fig.\ref{sb}) - from now on, we denote by $\overline{S_b}$ the closure of $S_b$ in the three-sphere $S^3$ (it is easy to see $\overline{S_b}$ is homeomorphic to $S^2$ - see the illustration in Fig.\ref{sb1}). We first prove the following fact:
\begin{proposition}\label{attrac1}
Let $p\in P$, $p=(a,b,c)$ be a parameter s.t. $\delta=\overline{W^s_{Out}}\cap \overline{S_b}$ is a closed curve on the sphere $\overline{S_b}$, which homotopic to the circle $S^1$ (as illustrated in Fig.\ref{sb1} or Fig.\ref{sb}). Then, the Rössler system corresponding to $p$ generates a compact attractor $A$ - moreover, $P_{Out}\not\in A$, and $A$ attracts the component $\Delta_{Out}$ of $W^u_{Out}$.
\end{proposition}
\begin{proof}
We first begin by analyzing the local dynamics of the vector field $F_p$ on $S_b$. To begin, note the normal vector to $S_b$ is $(0,0,1)$, and that by computation $F_p(x,y,-b)\bullet(0,0,1)=bx-b(x-c)=bc$ - per our assumption on the parameter space, since $a,b\in(0,1)$ and $c>1$, it follows that $F(x,y,-b)\bullet(0,0,1)>0$. Or, in other words, on every point of $s\in S_b$, $F_p(s)$ points inside the set $Q_b=\{(x,y,z)|z>-b\}$ (see the illustration in Fig.\ref{sb}). Moreover, by $P_{In}=(0,0,0)$ and by $P_{Out}=(c-ab,\frac{ab-c}{a},\frac{c-ab}{a})$ it is easy to see $P_{In},P_{Out}$ are both interior to $Q_b$. Consequentially, it follows the two dimensional unstable manifold $W^u_{In}$ for $P_{In}$ is also trapped inside $Q_b$.\\

\begin{figure}[h]
\centering
\begin{overpic}[width=0.4\textwidth]{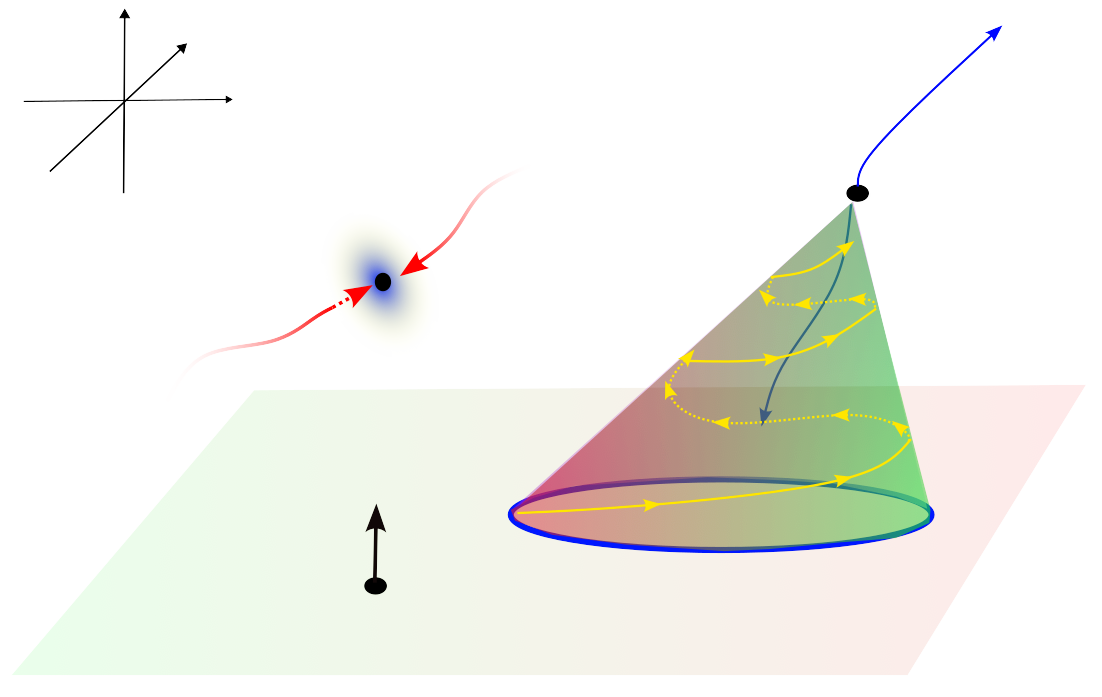}
\put(10,-20){$S_b$}
\put(400,490){$\Delta_{In}$}
\put(170,270){$\Gamma_{In}$}
\put(900,530){$\Gamma_{Out}$}
\put(800,450){$P_{Out}$}
\put(830,270){$W^s_{Out}$}
\put(300,430){$P_{In}$}
\put(545,250){$\Delta_{Out}$}
\put(565,80){$\delta$}
\put(220,530){$x$}
\put(160,610){$y$}
\put(100,635){$z$}
\end{overpic}
\caption[Fig31]{The plane $S_b$, the directions of $F_p$ on it, and the region $Q_b$ above it. In this illustration, $\overline{W^s_{Out}}\cap \overline{S_b}$ is the blue circle $\delta$, while $Q_1$ is the cone trapped between $S_b$ and $W^s_{Out}$. In this scenario $P_{In}\not\in Q_1$- it is easy to see $Q_b\setminus \overline W^s_{Out}$ includes two components, s.t. one of them is $Q_1$ (note $\Delta_{Out}\subseteq Q_1$).}
\label{sb}
\end{figure}
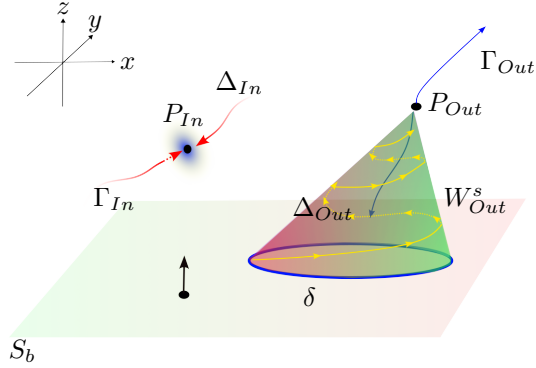

Now, let us consider the set $Q=Q_b\setminus \overline{W^s_{Out}}$ - by our assumption that $\overline{W^s_{Out}}\cap \overline{S_b}$ it is easy to see $Q$ is composed of two components, $Q_1$ and $Q_2$, as illustrated in Fig.\ref{sb1} and \ref{sb}. Since the one-dimensional unstable manifold $W^u_{Out}$ is transverse to $W^s_{Out}$ at the saddle-focus $P_{Out}$, by $W^s_{Out}=\Gamma_{Out}\cap\Delta_{Out}$ it follows $W^u_{Out}$ separates $\Gamma_{Out}$ and $\Delta_{Out}$ in $Q$ (see the illustration in Fig.\ref{sb1}). Therefore, let us denote by $Q_1$ the component of $Q$ s.t. $\Delta_{Out}\subseteq Q_1$ (see the illustration in Fig.\ref{sb1} and Fig.\ref{sb}).\\

\begin{figure}[h]
\centering
\begin{overpic}[width=0.4\textwidth]{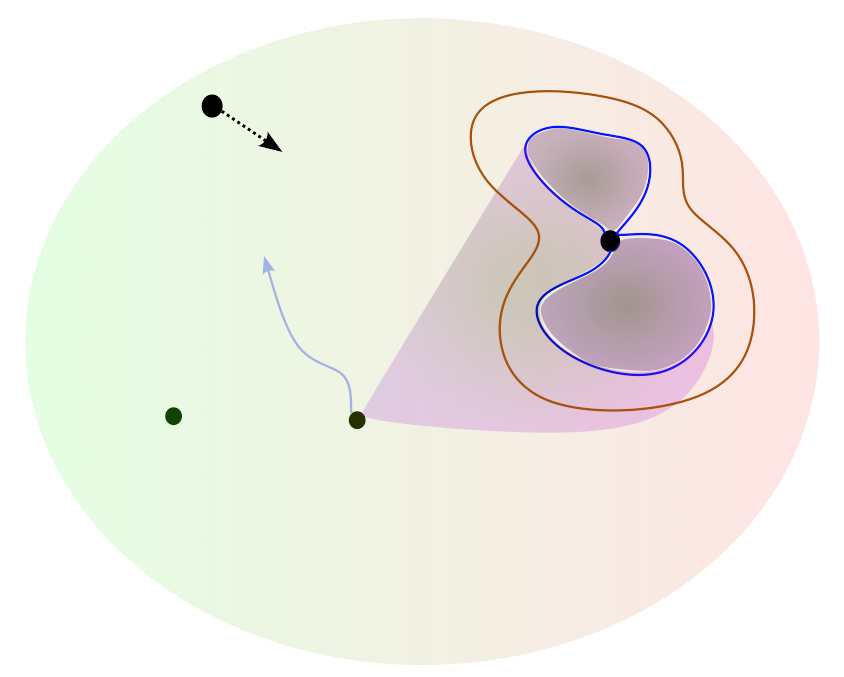}
\put(370,250){$P_{Out}$}
\put(150,270){$P_{In}$}
\put(720,560){$\delta$}
\put(750,490){$\infty$}
\put(830,270){$\gamma$}
\put(215,430){$\Delta_{Out}$}
\put(545,250){$W^s_{Out}$}
\end{overpic}
\caption[Fig31]{The set $\overline{S_b}$, sketched as a sphere in $S^3$ (with $\infty$ on it). The region $Q_b$ is the region trapped inside the sphere, while $W^s_{Out}$ is the purple surface. The set $Q_1$ is the component of ${Q_b}\setminus \overline{W^s_{Out}}$ which includes $P_{Out}$ (it also includes $\Delta_{Out}$, and in this scenario, we also have $P_{In}\in Q_1$). In this scenario, $\overline{W^s_{Out}}\cap \overline{S_b}$ is the blue curve $\delta$ which is homotopic (and not homeomorphic) to $S^1$, with singularities at $\infty$. The dark orange curve is $\gamma$.}
\label{sb1}
\end{figure}
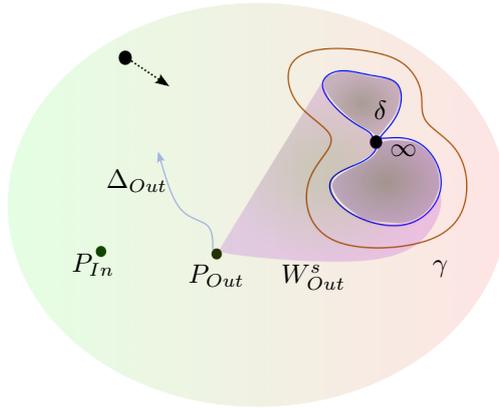

To continue, let $S_1$ denote a local cross-section transverse to $\Delta_{Out}$ - moreover, we choose $S_1$ to be sufficiently small s.t. it is transverse to the flow. Per our assumption that $\overline{W^s_{Out}}\cap \overline{S_b}$, it follows there exists a curve $\gamma\subseteq {S_b}\cap\partial Q_1$ (where both boundaries are taken in $\mathbf{R}^3$) s.t. the forward trajectory of any $s\in\gamma$ hits $S_1$ transversely (see the illustration in Fig.\ref{sb1} and Fig.\ref{sb2}). Now, consider the region $Q_A$, trapped between $S_b$ and the flow lines connecting $\gamma$ and $S_1$, as illustrated in Fig.\ref{sb2}. It is easy to see $Q_A$ is bounded in $\mathbf{R}^3$ - and furthermore, it is also easy to see $P_{Out}\not\in\overline Q_A$. By definition, since every $s\in\partial Q_A$ either lies on either $S_b$, $S_1$, or on some flow line connecting $\gamma$ and $S_1$, it follows the vector field is either tangent to $\partial Q_A$, or points inside $Q_A$ - consequentially, no trajectories can escape $Q_A$ under the flow. Therefore, recalling we denote the flow by $\phi^p_t$, $t\in\mathbf{R}$, we conclude $A=\overline{\lim_{t\to\infty}\phi^p_t(\partial Q_A)}$ is a compact, invariant set. In particular, $A$ is strictly interior to $Q_A$, hence $P_{Out}\not\in A$.\\

\begin{figure}[h]
\centering
\begin{overpic}[width=0.4\textwidth]{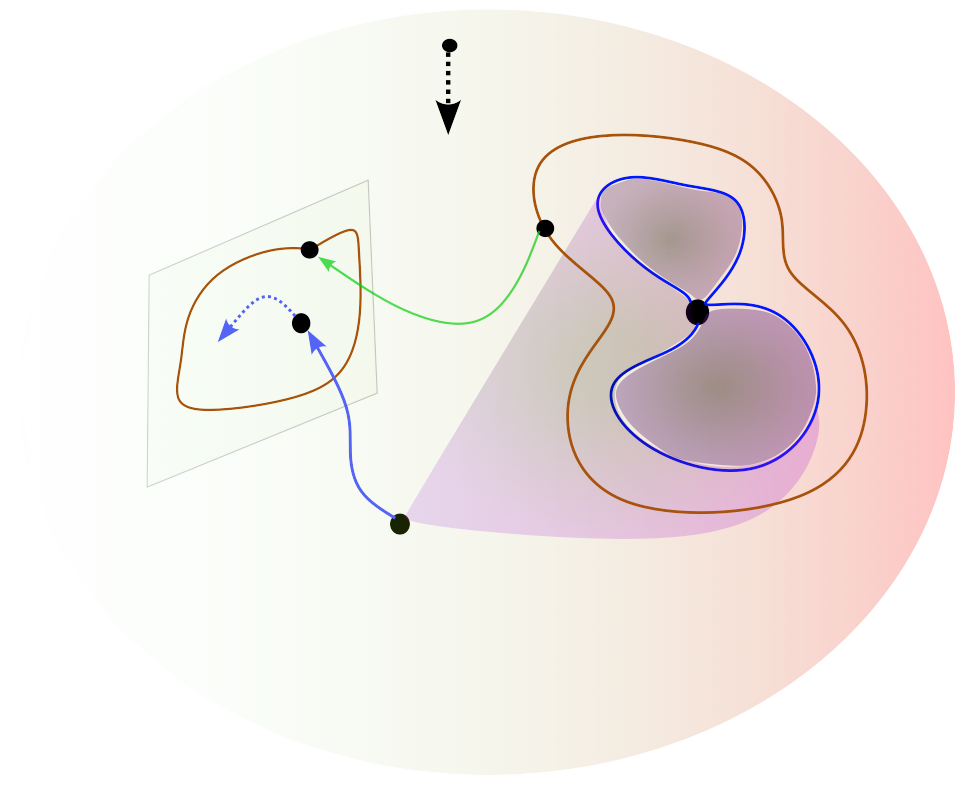}
\put(370,230){$P_{Out}$}
\put(720,560){$\delta$}
\put(750,490){$\infty$}
\put(830,270){$\gamma$}
\put(215,430){$\Delta_{Out}$}
\put(215,570){$\gamma_1$}
\put(385,600){$S_1$}
\put(545,200){$W^s_{Out}$}
\end{overpic}
\caption[Fig31]{The cross section $S_1$ - by the continuity of the flow, the trajectory of every initial condition on $\gamma$ flows to $\gamma_1$, some curve on $S_1$ (where $S_1$ is a cross-section transverse to the flow and to $\Delta_{Out}$). Consequentially, there exists a connected region $Q_A\subseteq Q_1$ from which no trajectories can escape.}
\label{sb2}
\end{figure}
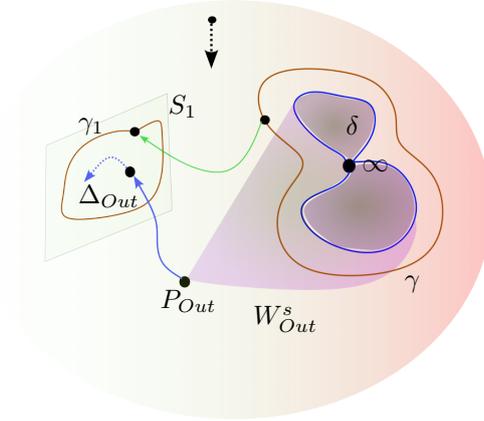

Since no trajectory can escape $Q_A$, by Th.6.1 in \cite{Chu}, it follows $A$ is an attractor for the flow. Having proven the existence of $A$, to conclude the proof of Prop.\ref{attrac1} it remains to prove the invariant manifold $\Delta_{Out}$ is attracted to $A$ - that is, that for all $s\in\Delta_{Out}$, $\overline{\lim_{t\to\infty}\phi^p_t(s)}\subseteq A$. That, however, is immediate - recalling $S_1\subseteq\partial Q_A$ is transverse to $\Delta_{Out}$ (see the illustration in Fig.\ref{sb2}), it follows that for every $s\in\Delta_{Out}$ there exists some $t_0\geq0$ s.t. $\phi^p_{t_0}(s)\in Q_A$. Since $A$ is the accumulation set under the flow for initial conditions in $Q_A$, Prop.\ref{attrac1} now follows.
\end{proof}
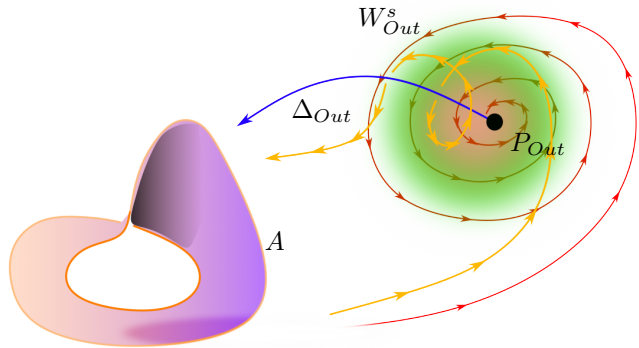
\begin{figure}[h]
\centering
\begin{overpic}[width=0.5\textwidth]{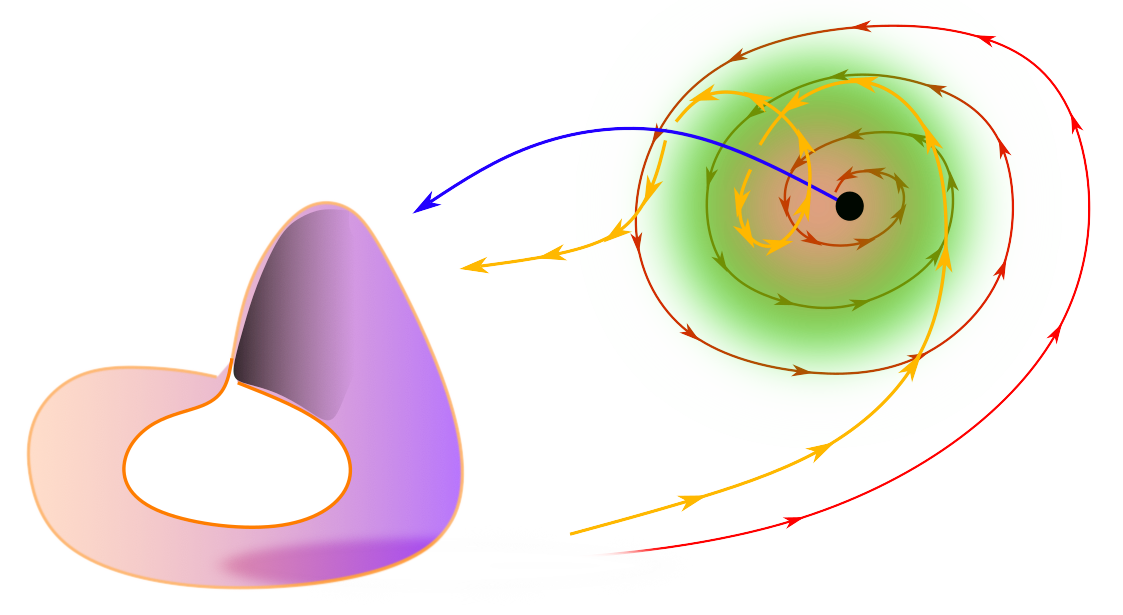}
\put(780,320){$P_{Out}$}
\put(550,510){$W^s_{Out}$}
\put(450,370){$\Delta_{Out}$}
\put(410,170){$A$}
\end{overpic}
\caption[Fig31]{Per Prop.\ref{attrac1}, the two-dimensional $W^s_{Out}$ shields trajectories from escaping to $\infty$, while $\Delta_{Out}$ repels them towards the attractor $A$ (and gets attracted to $A$ in itself).}
\label{sb3}
\end{figure}

Let $\kappa\subseteq P$ denote the parameter set s.t. for $v\in\kappa$, the Rössler system corresponding to $v$ satisfies the assumptions (and conclusions) of Prop.\ref{attrac1}. Even though it is not at all obvious from our arguments above that $\kappa$ is non-empty, from the numerical evidence of \cite{MBKPS}, \cite{G} and \cite{BBS}, it appears this is the case. In more detail, as observed numerically in \cite{BBS} and \cite{MBKPS} (among others), in many parameters the two-dimensional invariant manifold $W^s_{Out}$ appears to shield trajectories from escaping to $\infty$ - while $\Delta_{Out}$ pushes them towards the attractor, as illustrated in Fig.\ref{sb3}. Finally, before moving on, let us further remark the argument used to prove Prop.\ref{attrac1} in fact holds whenever $b,c>0$. That is, Prop.\ref{attrac1} holds even outside our parameter range - say, for example when $P_{Out}$ is a real saddle rather than a saddle focus. This correlates with \cite{BB2}, where it was observed that even for parameter values outside of $P$ the flow can still generate an attracting invariant set.\\

Having proven a sufficient condition for the existence of an attracting invariant set, we are now concerned with the converse - i.e., we prove the existence of a repeller for the flow. Taking a different approach to the one used to rove Prop.\ref{attrac1}, by direct qualitative analysis of the vector field we now prove:

\begin{lemma}
    \label{attrac3}
    Let $p\in P$, $p=(a,b,c)$ be parameter value - then, there exists an unbounded repeller $R$ for the corresponding Rössler system.
\end{lemma}
\begin{proof}
Consider the region $R=\{(x,y,z)|\dot{y}<0, \dot{x}>0,y>b+1\}$ (where $\dot{x}=-y-z$ and $\dot{y}=x+ay$ - see Eq.\ref{Field}). It is easy to see $R$ is trapped between the three surfaces $R_1=\{(x,-\frac{x}{a},z)|y>b+1, -z>\frac{x}{a}\}$, $R_2=\{(x,-z,z)|x-az<0,z<-b-1\}$ and $R_3=\{(x,b+1,z)|-z-b-1\geq0,x+a(b+1)<0\}$ (see the illustration in Fig.\ref{sb4}). By computation, the normal vectors to each surface are $r_1=(1,a,0)$, $r_2=(0,1,1)$ and $r_3=(0,-1,0)$ (respectively) - moreover it is easy to see $r_1,r_2$ and $r_3$ all point outside of $R$, as illustrated in Fig.\ref{sb4}.\\

\begin{figure}[h]
\centering
\begin{overpic}[width=0.3\textwidth]{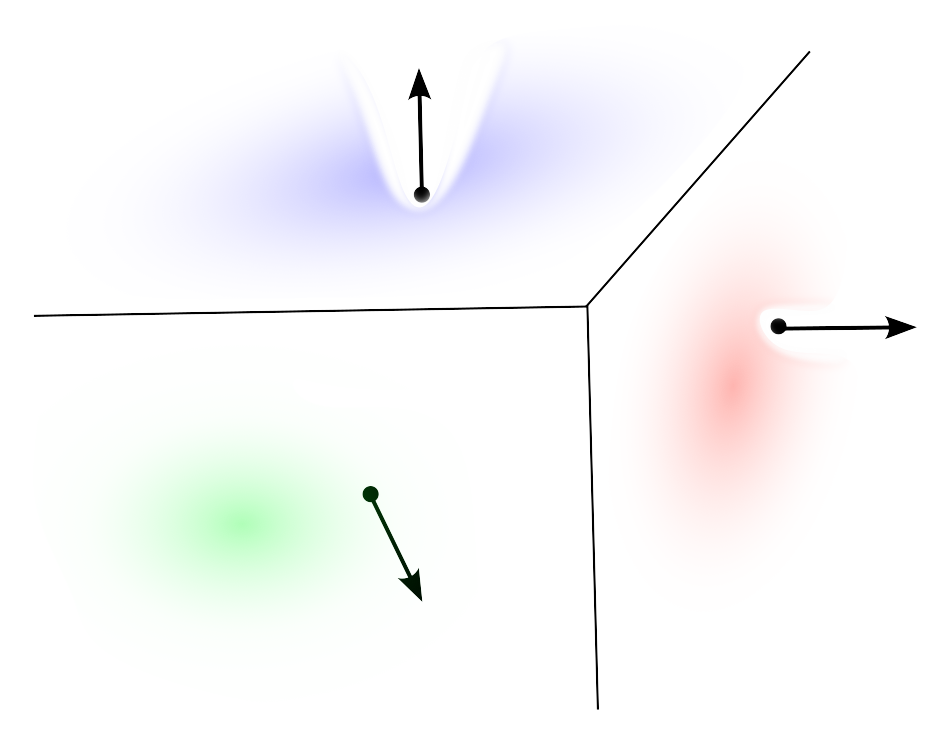}
\put(700,350){$R_1$}
\put(850,460){$r_1$}
\put(550,510){$R_2$}
\put(360,610){$r_2$}
\put(450,350){$R_3$}
\put(350,170){$r_3$}
\end{overpic}
\caption[Fig31]{The body $R$, trapped between $R_1,R_2$ and $R_3$ (where $r_1,r_2$ and $r_3$ are the respective normal vectors). As we prove, for $v_i\in R_i$, $F_p(v_i)\bullet r_i>0$ - i.e., $F_p$ points outside of $R$ throughout $\partial R$.}
\label{sb4}
\end{figure}
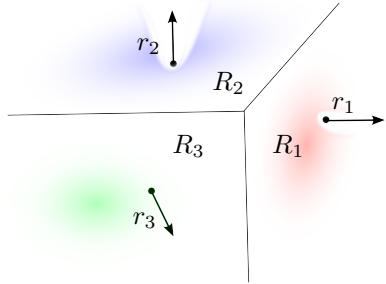

By computation, $F_p(x,-\frac{x}{a},z)\bullet(1,a,0)=\frac{x}{a}-z$, $F_p(x,b+1,z)(0,-1,0)=-(x+a(b+1))$, which implies that given $v_i\in R_i$, $i=1,3$, $F_p(v_i)\bullet r_i>0$. For $i=2$, let us note the equation $F_p(x,-z,z)\bullet(0,1,1)$ is positive precisely when $z<\frac{(b+1)x}{a+c-x}$. Set $f(x)=\frac{(b+1)x}{a+c-x}$ - it is easy to see $f$ is non-decreasing on $x\in(-\infty,a+c)$ and that $\lim_{x\to-\infty}f(x)=-b-1$ - which, by $z\leq-b-1<f(x)$, $x\in(-\infty,a+c)$ implies that for all $(x-z,z)\in R_2$ we have $z<\frac{(b+1)x}{a+c-x}$. Therefore, again, for $v_3\in R_3$, $F_p(v_3)\bullet(0,1,1)>0$. Consequentially, it follows $F_p$ points outside of $R$ throughout $\partial R$, which implies $R$ is a repeller. Moreover, since $R$ is unbounded the assertion follows.
\end{proof}
\begin{remark}
Recall we denote the flow corresponding to $F_p$ by $\phi^p_t$, $t\in\mathbf{R}$. With just a little more work, one can easily prove that for any $v\in R$ and all $t<0$, $\phi^p_t(v)\in R$ - and that there are no fixed points in $R$. Since by definition $R\subseteq\{\dot{y}\leq0\}$, for every $v\in R$ we have $\lim_{t\to-\infty}\phi^p_t(v)=\infty$. 
\end{remark}
\begin{remark}
    In addition to the well-known attractor, the existence of repelling, invariant sets for the Rössler system was observed numerically in \cite{BBS}. However, in \cite{BBS} the said set is composed of periodic trajectories.
\end{remark}
\begin{figure}[h]
\centering
\begin{overpic}[width=0.5\textwidth]{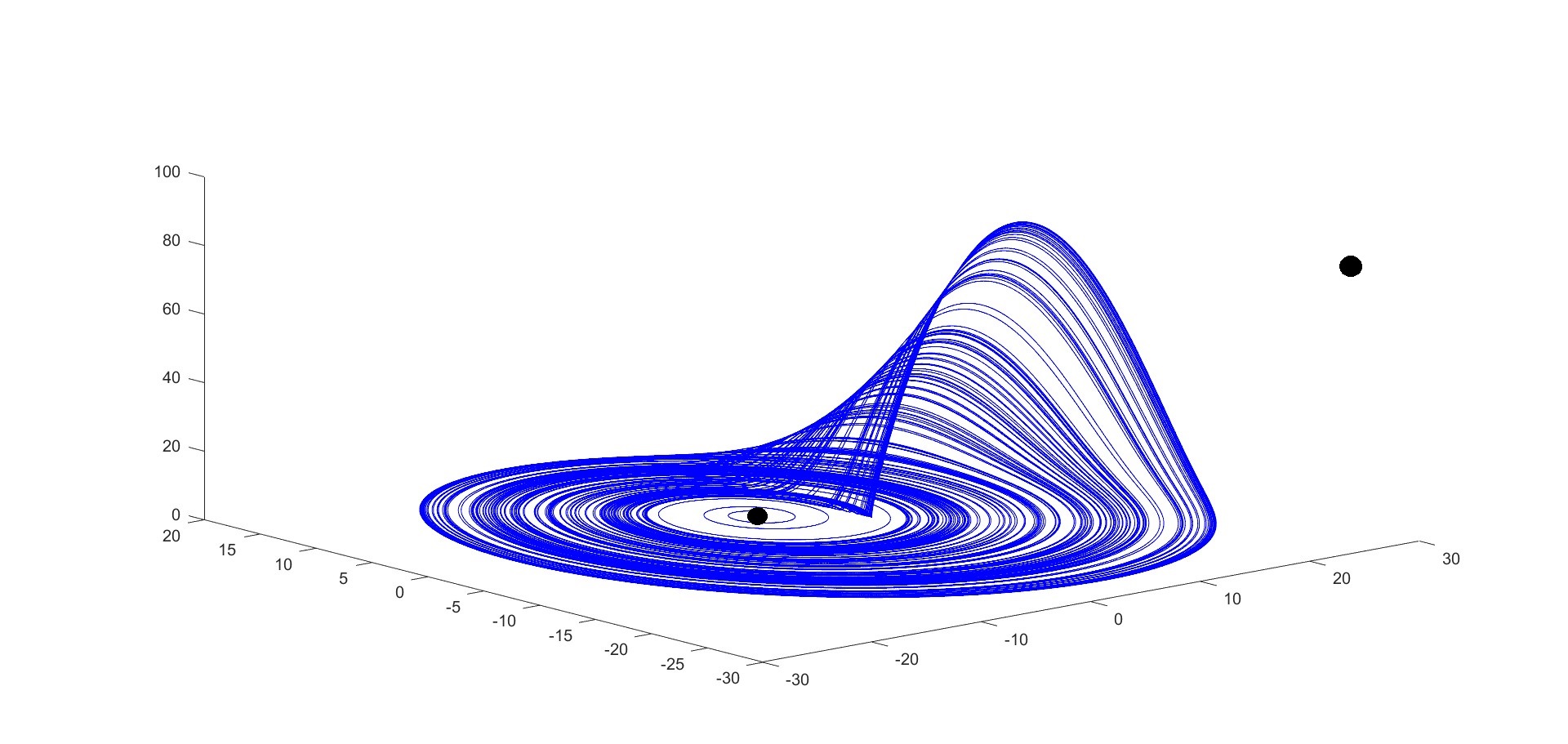}
\end{overpic}
\caption[Fig31]{The attractor of the Rössler system given by Eq.\ref{Vect} at $(a,b,c)=(0.8.0.2,15.7)$, sketched in Matlab. The fixed-points are represented as black dots.}
\label{cross4}
\end{figure}

Having proven Prop.\ref{attrac1} and Lemma \ref{attrac3}, we now summarize our results as follows:

\begin{theorem}
    \label{attrac} Let $p=(a,b,c)$ be some parameter in the parameter space $P$. Then, the corresponding Rössler system satisfies the following:
    \begin{enumerate}
        \item There exists an unbounded repeller for the flow.
        \item Recall $S_b=\{(x,y,-b)|x,y\in\mathbf{R}\}$ and that $W^s_{Out}$ denotes the two-dimensional stable manifold of the saddle focus $P_{Out}$. Then, provided $\delta=\overline{W^s_{Out}}\cap \overline{S_b}$ is a closed curve on $\overline{S_b}$ homotopic to $S^1$ (where both closures are taken in $S^3$), we have the following: 
        
        \begin{itemize}
            \item There exists a compact attractor for the flow, $A$. Moreover, $P_{Out}\not\in A$.
            \item A component $\Delta_{Out}$ of $W^u_{Out}$ is attracted to $A$, as illustrated in Fig.\ref{sb3}.
            \item The attractor $A$ is robust - i.e., it persists under sufficiently small $C^1$ perturbations of the vector field $F_p$.
        \end{itemize}
    \end{enumerate}
\end{theorem}
\begin{proof}
    By Prop.\ref{attrac1} and Lemma \ref{attrac3}, the only thing we need prove is the robustness of $A$. To do so, recall $A$ is compact - and that its basin of attraction, $Q_A$ (as defined in the proof of Prop.\ref{attrac1}) is pre-compact in $\mathbf{R}^3$. It is easy to see that by deforming $Q_A$ with the flow we generate a pre-compact set $B_A\subseteq Q_A$ s.t. $A$ is strictly interior to $B_A$, and the vector field $F_p$ points inside $B_A$ throughout $\partial B_A$ - moreover, since $S^2$ is homeomorphic to $S^2$ we can also ensure $\partial B_A$ forms a smooth surface, diffeomorphic to $S^2$. Consequentially, by the pre-compactness of $B_A$ it follows that for any sufficiently small $C^1$ perturbation of $F_p$, denoted by $F$, the vector field $F$ points inside $B_A$ throughout $\partial B_A$ and Th.\ref{attrac} now follows.
\end{proof}
\begin{figure}[h]
\centering
\begin{overpic}[width=0.5\textwidth]{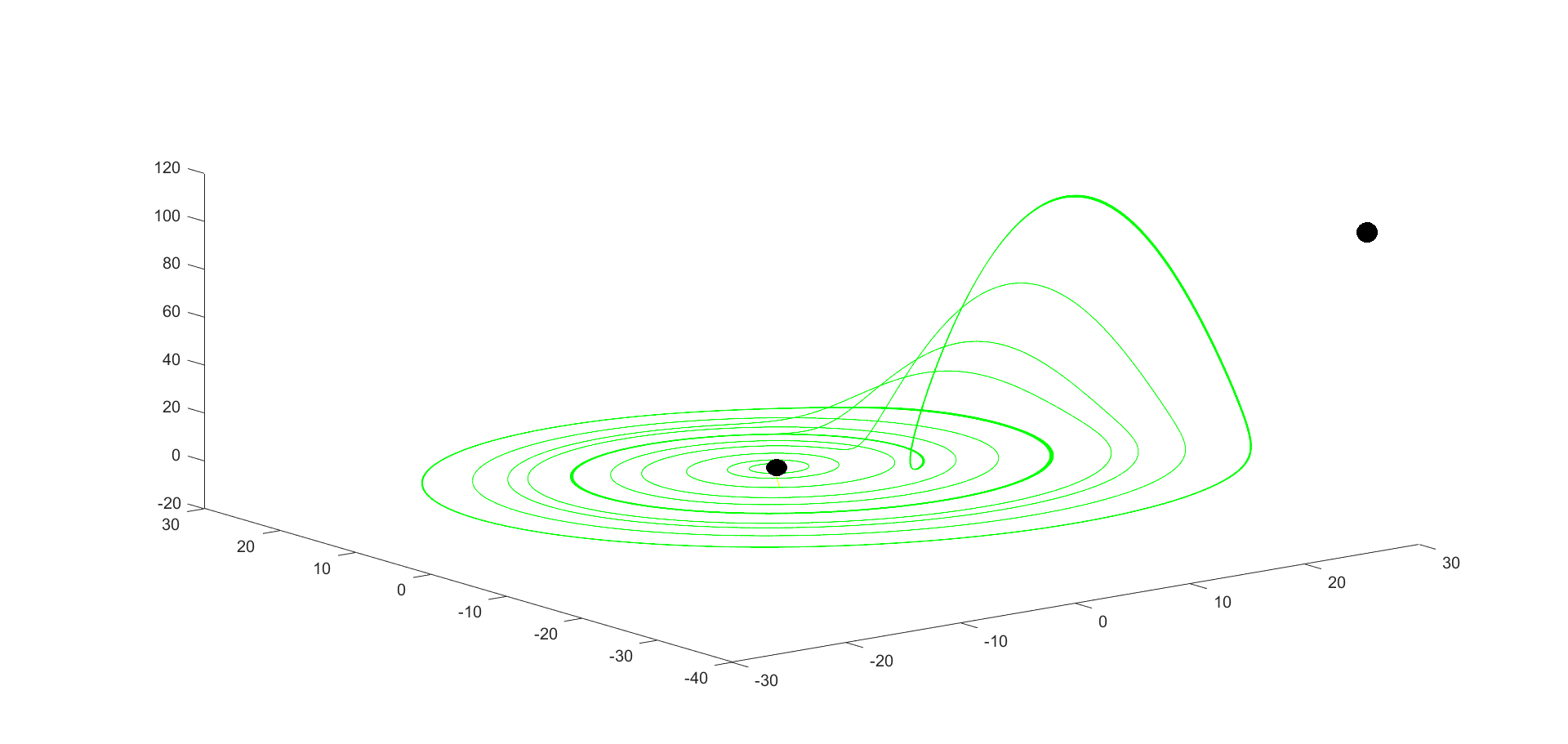}
\end{overpic}
\caption[Fig31]{The attractor of the Rössler system given by Eq.\ref{Vect} at $(a,b,c)=(0.8.0.2,18.33)$, sketched in Matlab (the fixed points correspond to the black dots) - at these parameters, the attractor appears to be a stable periodic trajectory (with a very long period).}
\label{cross7}
\end{figure}

Before we conclude this section, let us discuss several remarks arising from Th.\ref{attrac}. The first is that Th.\ref{attrac} correlates closely with the numerical results on the Rössler attractor. As observed numerically, it is the invariant manifolds of $P_{Out}$ which are responsible for the creation of the attractor (see, for example, \cite{MBKPS}, \cite{BBS}, and \cite{BB2} - among others) - that is, the two-dimensional stable manifold $W^s_{Out}$ shields trajectories from escaping to $\infty$, while a bounded component of $W^u_{Out}$ (i.e., $\Delta_{Out}$) propels them towards the attractor (see the illustration in Fig.\ref{sb3}). It is easy to see Th.\ref{attrac} (or more precisely, Prop.\ref{attrac1}) is an analytic counterpart of these observations.\\

A second (and final) remark is that Th.\ref{attrac} is essentially an existence theorem - that is, it does not teach us anything about the dynamics on the attractor $A$. In particular, Th.\ref{attrac} does not give us enough information to determine the chaoticity of the attractor. As observed numerically, the Rössler attractor, for whichever $(a,b,c)$ parameters it may exist in $P$, need not be chaotic. In fact, one of the widely observed facts about the Rössler attractor is that at some regions of $P$ it is chaotic - yet in others it is a stable, attracting periodic trajectory. This leads us to ask the following question: \textbf{can we analytically describe the evolution of the Rössler attractor from order to chaos?}\\

In the next section we will give a partial answer to this question. In more detail, we will prove that at least around trefoil parameters (see Def.\ref{def32}), the dynamical complexity of the Rössler system can be described by the discrete-time dynamics of a family of one-dimensional maps: the quadratic family, i.e., $p_c(x)=x^2+c$, $c\in[-2,\frac{1}{2}]$.
\section{One-dimensional first return maps in the Rössler system}
From now on unless said otherwise, $v=(a,b,c)$ would always denote a parameter in $P$, while $F_v$ would always denote the vector field generating the corresponding Rössler system (see Eq.\ref{Field}). Similarly, from now on $p=(a,b,c)$ would always denote a trefoil parameter for the Rössler system and $F_p$ would always denote the corresponding vector field (see Def.\ref{def32}). Additionally, recall that given a $C^1$-vector field of $\mathbf{R}^3$, $F$, its non-wandering set is the collection of initial conditions whose trajectory is not attracted to $\infty$ (see Def.\ref{nowandering}) - in particular, the non-wandering set includes every periodic trajectory generated by $F$. Our main objective in this section is to prove Th.\ref{polytheorem} - where we prove that given any parameter $v\in P$ sufficiently close to trefoil parameter, the dynamics of the Rössler system corresponding to $v$ can be described by those of a quadratic polynomial.\\

This section is organized as follows - we begin by proving the dynamics of the Rössler system at trefoil parameters on a certain invariant set can be reduced to those of $p_{-2}(x)=x^2-2$ on the interval $[-2,2]$ (see Cor.\ref{misiu}). Following that, we give a rough intuition to Th.\ref{polytheorem}, after which we prove the said Theorem. Unlike Th.\ref{attrac}, the proof of Th.\ref{polytheorem} would be heavily based on ideas from Th.\ref{th31} and Th.\ref{th41} - that is, it will be strongly based on describing the flow dynamics (or more precisely, those of the first-return map) by symbolic means.\\

\begin{figure}[h]
\centering
\begin{overpic}[width=0.5\textwidth]{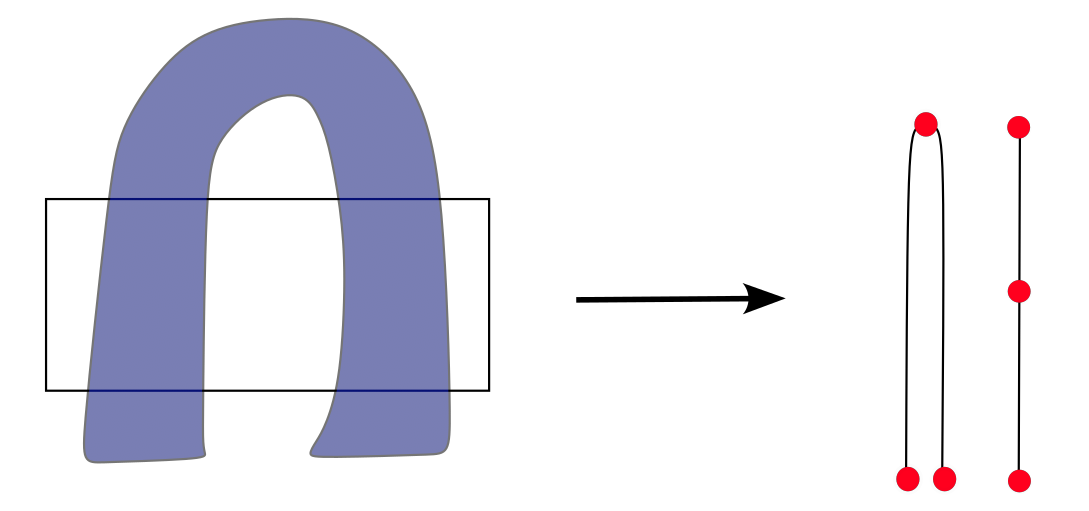}
\put(10,60){$A$}
\put(970,30){$2$}
\put(10,300){$C$}
\put(970,340){$-2$}
\put(800,380){$p_{-2}(0)$}
\put(970,190){$0$}
\put(750,-30){$p_{-2}(2)=p_{-2}(-2)$}
\put(455,60){$B$}
\put(455,300){$D$}
\put(300,0){$H(AB)$}
\put(70,0){$H(CD)$}
\end{overpic}
\caption[Fig31]{Crushing a Smale Horseshoe map $H:ABCD\to\mathbf{R}^2$ to an unimodal map which covers the interval twice - i.e, to $p_{-2}$ on $[-2,2]$.}
\label{crush}
\end{figure}
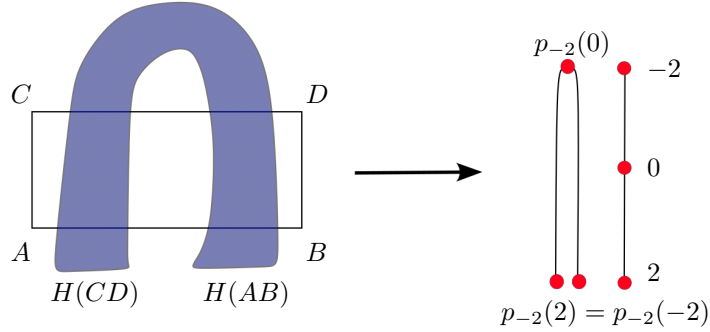

To begin, we first recall some formalism and facts about polynomial dynamics from Section $1.2$. Specifically, recall we denote the quadratic family by $p_c(x)=x^2+c,c\in[-2,\frac{1}{4}]$, and that $\sigma:\{1,2\}^\mathbf{N}\to\{1,2\}^\mathbf{N}$ denotes the one-sided shift. As explained at the beginning of Section $1.2$, for $ c\in[-2,\frac{1}{4}]$, the polynomial $p_c$ admits a maximal bounded interval $[x_{2,c},x_{1,c}]=V_c$, s.t. the following is satisfied:
\begin{itemize}
    \item $p_c(V_c)\subseteq V_c$.
    \item $0$ is interior to $V_c$.
    \item For $x\not\in V_c$, $|p^n_c(x)|\to\infty$.
\end{itemize}
See the illustration in Fig.\ref{logistic}. Moreover, recall we denote by $I_c$ as the maximal invariant set of $p_c$ in $V_c\setminus\{0\}$. Denoting $[x_{2,c},0)$ by $2$ and $(0,x_{1,c}]$ by $1$, in Section $1.2$ we showed there exists a continuous map $\xi_c:I_c\to\{1,2\}^\mathbf{N}$ s.t. $\xi_c\circ p_c=\sigma\circ\xi_c$ (see Section $1.2$ for more details and references). In particular, for $c=-2$ we have $V_{-2}=[-2,2]$ - and moreover, $p_{-2}$ covers $[-2,2]$ twice, i.e., $p_{-2}([-2,2])=[-2,2]$ (see Lemma \ref{polyhorse}).\\

Now, let $\kappa\subseteq\{1,2\}^\mathbf{N}$ denote the set of all symbols $s\in\{1,2\}^\mathbf{N}$ which are not strictly pre-periodic to the constant $\{1,1,1,...\}$. Per Lemma \ref{polyhorse}, for $c=-2$, $\xi_{-2}:I_{-2}\to\kappa$ is a homeomorphism, satisfying $\xi_{-2}\circ p_{-2}\circ\xi^{-1}_{-2}=\sigma$. Consequentially, we conclude the dynamics of $p_{-2}$ on $[-2,2]$ are a singular model of the Smale-Horseshoe map in the following sense: given a Smale Horseshoe map $H:ABCD\to\mathbf{R}^2$, when we collapse $H$ by a homotopy to an interval map, we get a unimodal map which covers the interval twice - i.e., we get $p_{-2}$ on $[-2,2]$ (up to some conjugation - see the illustration in Fig.\ref{crush}).\\

\begin{figure}[h]
\centering
\begin{overpic}[width=0.45\textwidth]{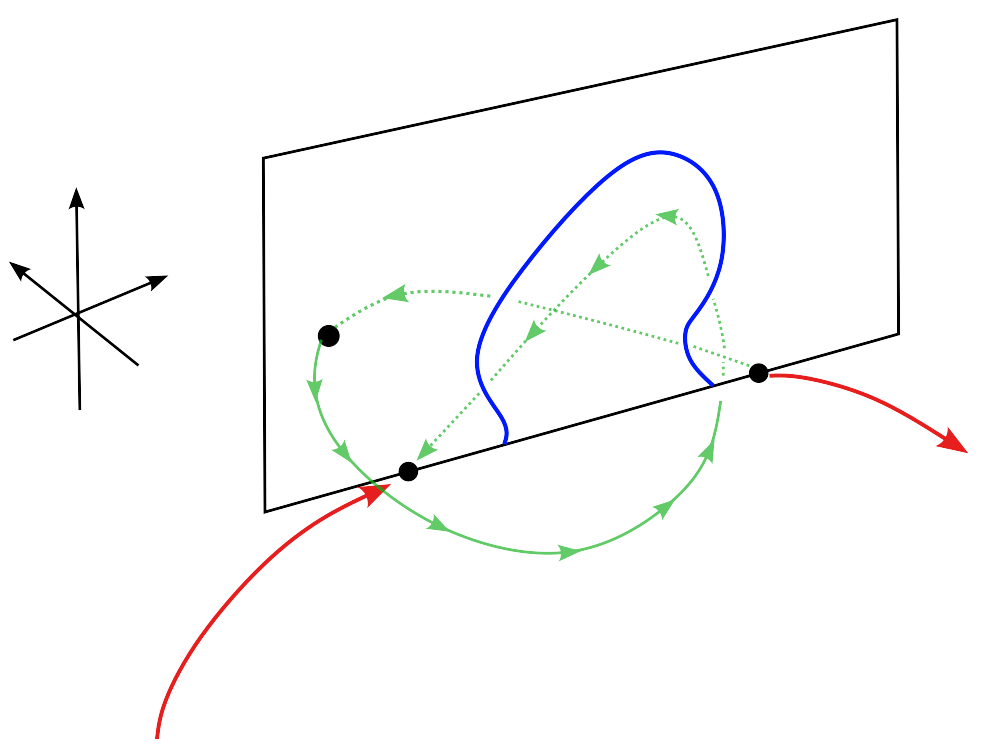}
\put(65,575){$z$}
\put(320,550){$D_1$}
\put(440,430){$\rho$}
\put(265,400){$P_0$}
\put(620,550){$D_2$}
\put(800,480){$U_p$}
\put(-15,505){$y$}
\put(170,470){$x$}
\put(425,240){$P_{In}$}
\put(575,140){$\Theta$}
\put(200,50){$\Gamma_{In}$}
\put(750,330){$P_{Out}$}
\put(910,360){$\Gamma_{Out}$}
\end{overpic}
\caption[Fig31]{The symbolic dynamics at trefoil parameters, generated by $D_1$ and $D_2$, the components of $\overline{U_p}\setminus\rho$. $\Theta$ denotes the bounded heteroclinic trajectory (see Def.\ref{def32}), while $\Gamma_{In}$ and $\Gamma_{Out}$ denote the unbounded heteroclinic trajectories (see Th.\ref{th21}).}
\label{crossec}
\end{figure}
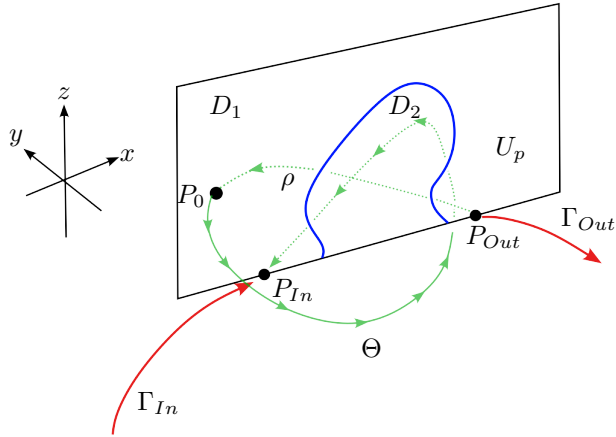

We now recall some facts about the dynamics of the Rössler system at trefoil parameters - namely, Cor.\ref{TR} and Th.\ref{th31}. In more detail, recall there exists a universal cross-section for the flow, the half-plane $U_p$ (see Cor.\ref{TR} and the illustration in Fig.\ref{crossec}), and that we denote the first-return map by $f_p:\overline{U_p}\to\overline{U_p}$ (wherever defined) - by Cor.\ref{TR}, every trajectory of the non-wandering set for Rössler system at trefoil parameters intersects $U_p$ transversely infinitely many times. Per Th.\ref{th31}, there exists a bounded curve $\rho\subseteq \overline{U_p}$, a bounded, $f_p$-invariant set $Q\subseteq\overline{U_p}\setminus\rho$ and a map $\pi:Q\to\{1,2\}^\mathbf{N}$ s.t. the following holds:

\begin{itemize}
    \item $Q$ is a subset of the non-wandering set for the vector field $F_p$.
    \item $\pi\circ f_p=\sigma\circ \pi$.
    \item $f_p$ and $\pi$ are both continuous on $Q$.
    \item $\pi(Q)$ includes every periodic $s\in\{1,2\}^\mathbf{N}$.
    \item If $s\in\{1,2\}^\mathbf{N}$ is periodic of minimal period $k$, $\pi^{-1}(s)$ includes at least one periodic point of minimal period $k$ for $f_p$.
\end{itemize}

See the illustration in Fig.\ref{part1} and Fig.\ref{crossec}. Combined with Lemma \ref{polyhorse} from Section $1.2$, by the discussion above we conclude:
\begin{corollary}\label{misiu}
    Let $p\in P$ be a trefoil parameter, let $Q$ be as above, and let $I_{-2}$ be as in Lemma \ref{polyhorse}. Set $\psi=\xi_{-2}^{-1}\circ\pi$ (with $\xi_{-2}$ as in Lemma \ref{polyhorse} and $\pi$ as above) - then, the function $\psi:Q\to I_{-2}$ is continuous, satisfies  $\psi\circ f_p=p_{-2}\circ\psi$ - and moreover, $\psi(Q)$ includes every periodic orbit for $p_{-2}$ in the interval $V_{-2}$.
\end{corollary}
\begin{remark}
The function $p_{-2}$ is not special - in fact, we can replace it with any unimodal map $f:[0,1]\to[0,1]$ s.t. $f$ covers $[0,1]$ twice (for example - the tent map). 
\end{remark}
\begin{remark}
    In \cite{MBKPS}, it was observed the dynamics of the first-return maps for the Rössler attractor at homoclinic parameters behaved like those of Misiurewicz maps - that is, unimodal maps for which the critical point lies away from $\overline{\cup_{n\geq1}p^n_c(0)}$ (see \cite{Mis} for more details). As $0$ is a pre-periodic point for $p_{-2}(x)=x^2-2$, Cor.\ref{misiu} is a possible explanation for this observation.
\end{remark}
\begin{figure}[h]
\centering
\begin{overpic}[width=0.5\textwidth]{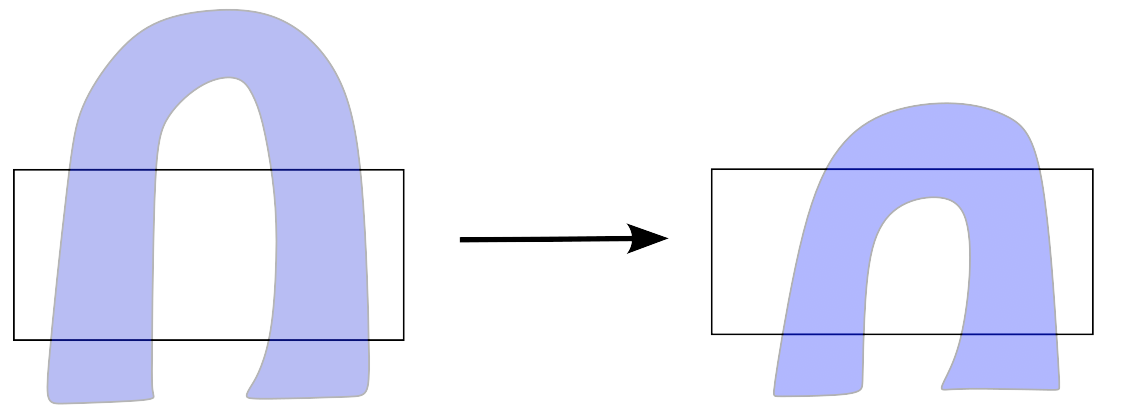}
\put(-10,30){$A$}
\put(960,30){$B$}
\put(-10,230){$C$}
\put(610,30){$A$}
\put(610,230){$C$}
\put(960,230){$D$}
\put(650,-20){$H'(CD)$}
\put(820,-20){$H'(AB)$}
\put(355,30){$B$}
\put(355,230){$D$}
\put(200,-30){$H(AB)$}
\put(20,-30){$H(CD)$}
\end{overpic}
\caption[Fig31]{Isotoping a Smale Horseshoe map $H:ABCD\to\mathbf{R}^2$ to a rectangle map $H':ABCD\to\mathbf{R}^2$ which is not a horseshoe map. $H'$ includes less periodic orbits than $H$..}
\label{crush4}
\end{figure}
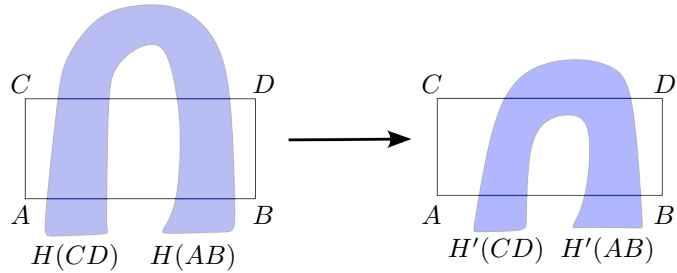

Having proven Cor.\ref{misiu}, we are now ready to state prove Th.\ref{polytheorem}. Even though the proof is relatively long due to several technicalities, the motivation behind it is rather simple. To introduce it, recall the motivation behind Cor.\ref{misiu} was that the dynamics of the first-return map at trefoil parameter essentially those of a horseshoe map (see Fig.\ref{fig16}) - therefore, these dynamics can be homotopically collapsed to an interval map which covers the interval twice (see Fig.\ref{crush} for an illustration). It is also easy to see that given a trefoil parameter $p\in P$, the dynamics of its corresponding vector field $F_p$ are structurally unstable under arbitrarily small $C^1$ perturbations of $F_p$ - therefore, as we perturb the dynamics of the Rössler system from those of some trefoil parameter $p$ to those some nearby parameter $v$, we expect the Rössler system corresponding to $v$ to include less periodic dynamics (at least when compared to those at trefoil parameters). Or, in other words, we expect the symbolic dynamics of the first-return map to include less periodic symbols. Therefore, inspired by Cor.\ref{misiu} and Th.\ref{maximal}, we would expect the dynamics of the perturbed flow to be similar to those of $p_c$ on the interval $V_c$ for some $c\in(-2,\frac{1}{4}]$ (see Fig.\ref{crush2}).\\

\begin{figure}[h]
\centering
\begin{overpic}[width=0.5\textwidth]{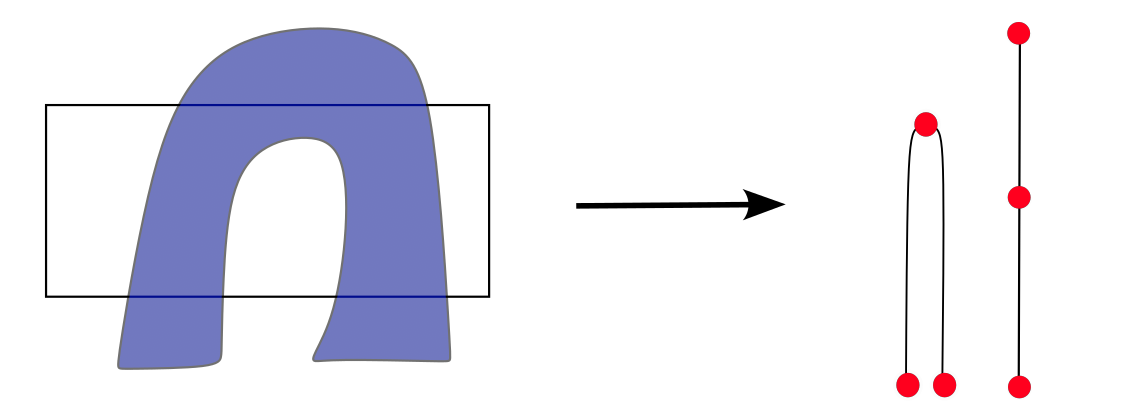}
\put(10,60){$A$}
\put(970,30){$x_2$}
\put(10,300){$C$}
\put(970,340){$x_1$}
\put(780,300){$f(0)$}
\put(970,190){$0$}
\put(750,-30){$f(x_1)=f(x_2)$}
\put(455,60){$B$}
\put(455,300){$D$}
\put(300,0){$H'(AB)$}
\put(70,0){$H'(CD)$}
\end{overpic}
\caption[Fig31]{Collapsing a rectangle map $H':ABCD\to\mathbf{R}^2$ to a unimodal map $f$ which does not cover the interval twice - consequentially, its symbolic dynamics are expected to not include all the periodic symbols in $\{1,2\}^\mathbf{N}$.}
\label{crush2}
\end{figure}
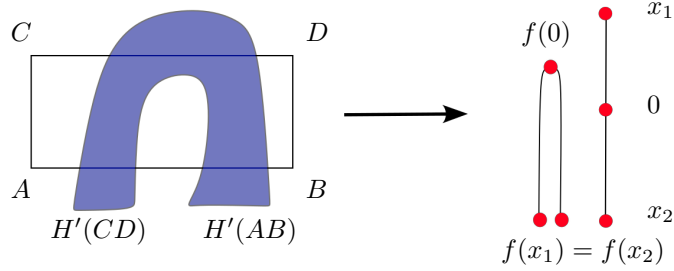

In order to make this heuristic into a rigorous proof, first recall that given any parameter $v\in P$, we denote by $f_v:\overline{U_v}\to\overline{U_v}$ the first-return map generated by the corresponding Rössler system - while $F_v$ always denotes the vector field. Additionally, recall that by Cor.\ref{TR} the trajectory of periodic trajectory in the non-wandering set for $F_v$ intersects transversely with the half-plane $U_v$. Furthermore, recall that as described in Section 1.1, when we vary the parameter $v$ to some $v'$ in the parameter space $P$, the cross-section $\overline{U_v}$ is continuously deformed to the cross-section $\overline{U_{v'}}$ - consequentially, given a trefoil parameter $p\in P$, as we smoothly deform the vector field $F_p$ to $F_v$ through the parameter space, the curve $\rho\subseteq\overline{U_p}$ is continuously deformed to defined above $\rho_v\subseteq \overline{U_v}$ (see the discussion immediately before Cor.\ref{invariant}, and the illustrations in Fig.\ref{part1}). Now, recall that for any parameter $v\in P$, we denote by $I_v$ the collection of initial conditions $x$ in $\overline{U_v}\setminus\rho_v$ satisfying the following:

\begin{itemize}
    \item The trajectory of $x$ is not attracted to $\infty$.
    \item For all $k>0$, $f^k_v(x)\in\overline{U_v}\setminus\rho_v$.
\end{itemize}

Recalling the discussion preceding Cor.\ref{invariant}, we know $\overline{U_v}\setminus\rho_v$ is composed of two components, $D_{1,v}$ and $D_{2,v}$ - both of which vary continuously with $v$ (see the illustration in Fig.\ref{part1}). Therefore, by Cor.\ref{invariant} and Th.\ref{th41} we have the following:

\begin{itemize}
    \item The set $I_v$ is non-empty. Moreover, it is a subset of the non-wandering set for $F_v$ (see Cor.\ref{invariant}).
    \item There exists a map $\pi_v:I_v\to\{1,2\}^\mathbf{N}$ s.t. $\pi_v\circ f_v=\sigma\circ\pi_v$.
    \item Given any periodic $s\in\{1,2\}^\mathbf{N}$ of minimal period $k$, provided $v$ is sufficiently close to the trefoil parameter $p$, $\pi^{-1}_v(s)$ includes $x_s$, a periodic point for $f_v$ of minimal period $k$.
    \item With the notations above, $f_v,...,f^k_v$ are all continuous on $x_s$. Similarly, $\pi_v$ is continuous on $x_s,...,f^{k-1}_v(x_s)$.
\end{itemize}

 \begin{figure}[h]
\centering
\begin{overpic}[width=0.6\textwidth]{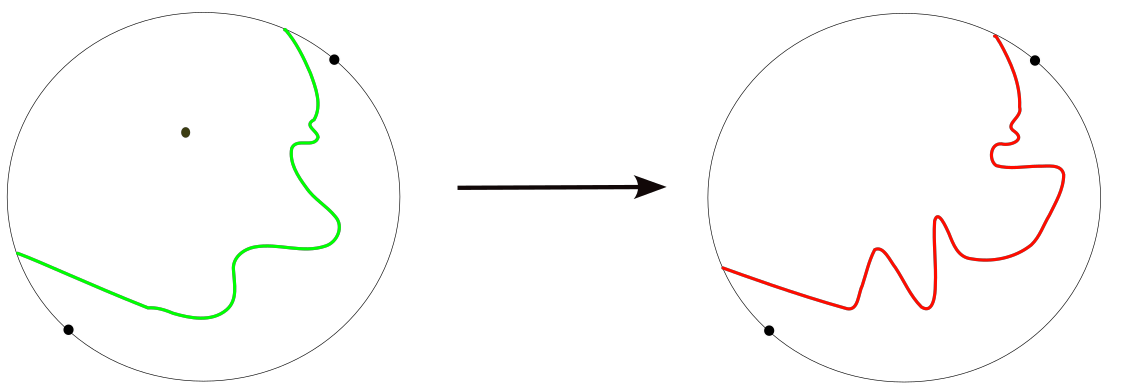}
\put(0,50){$P_{In}$}
\put(620,50){$P_{In}$}
\put(170,240){$P_0$}
\put(230,200){$\rho$}
\put(840,200){$\rho_v$}
\put(70,150){$D_2$}
\put(720,250){$D_{2,v}$}
\put(220,80){$D_1$}
\put(860,70){$D_{1,v}$}
\put(320,300){$P_{Out}$}
\put(930,300){$P_{Out}$}
\end{overpic}
\caption[Fig32]{the deformation of $U_p$ where $p$ is a trefoil parameter (on the left) to $U_v$ (on the right) - for simplicity, $U_p,U_v$ are drawn as discs rather than half-planes. As can be seen, the curve $\rho$ is deformed to $\rho_v$.}
\label{part1}
\end{figure}

We are now ready to state and prove Th.\ref{polytheorem}. With these ideas in mind, we prove:
\begin{theorem}\label{polytheorem}
Let $p\in P$ denote a trefoil parameter for the Rössler system. Then, there exists a function  $\Pi:P\rightarrow [-2,\frac{1}{4}]$, s.t. setting $d=\Pi(v)$, $v=(a,b,c)\in P$ we have the following:
\begin{itemize}
 
    \item  When $v\to p$, $\Pi(v)\to -2$ - i.e., $\Pi$ is continuous at trefoil parameters.
    \item Let $f_v:\overline{U_v}\to \overline{U_v}$ denote the first-return map corresponding to a given $v\in P$. Then, there exist an $f_v-$invariant $J_v\subseteq I_v$, a bounded, $p_d$-invariant $J_d\subseteq\mathbf{R}$, and a surjective $\zeta_v:J_v\rightarrow J_d$, s.t. $\zeta_v\circ f_v=p_d\circ\zeta_v$.
    \item  Given any $n>0$, provided the parameter $v$ is sufficiently close to $p$, we have the following:
    \begin{enumerate}
        \item  the set $J_v$ includes $\Omega_1,...,\Omega_n$ - $n$ distinct periodic orbits for $f_v$.
        \item Both $f_v,\zeta_v$ are continuous on $\Omega_i$, $1\leq i\leq n$.
        \item $\zeta_v(\Omega_i)=P_i$, $1\leq i\leq n$ are periodic orbits for $p_d$ - moreover, $\Omega_i$ and $P_i$ have the same minimal period.
    \end{enumerate}
    In other words, as $v\to p$, the factor map $\zeta_v$ becomes increasingly more continuous.

\end{itemize}
\end{theorem}
Before moving on to the proof, let us remark the formalism above has the following meaning - the closer a given Rössler system is to the dynamics of a trefoil parameter, the more its behavior around periodic trajectories looks like that of a deformed, suspended quadratic polynomial. In other words, the smaller $||v-p||$ is (where $p$ is a trefoil parameter), the better the flow dynamics on the non-wandering set are described by the semi-flow generated by suspending $p_d(x)=x^2+d$, $d=\Pi(v)$ - see the illustrations at Fig.\ref{match1} and Fig.\ref{template}. 
\begin{proof}
From now on $p\in P$ would always denote a trefoil parameter for the Rössler system - while $F_p$ always denoting the corresponding vector field. Similarly, given any $v\in P$ (not necessarily a trefoil parameter), we will always denote by $F_v$ the corresponding vector field (see Eq.\ref{Field}). Additionally, from now on $\sigma:\{1,2\}^\mathbf{N}\to\{1,2\}^\mathbf{N}$ would always denote the one-sided shift. Before giving a sketch of proof, we first make some general remarks about the symbolic dynamics for the vector field $F_v$ (as defined at the discussion above).\\

To begin, for every $v\in P$, set $S_v=\pi_v(I_v)$, and define $Per(v)\subseteq S_v$ as the set of periodic symbols in $S_v$. Recall that by Cor.\ref{invariant}, for every $v$ the constant $\{1,1,1...\}$ is in $Per(v)$ - i.e., for every $v\in P$, $Per(v)\ne\emptyset$. Moreover, by Th.\ref{th41} we immediately conclude:

\begin{corollary}
    \label{dista1}
    Let $p\in P$ be a trefoil parameter, and let $s\in\{1,2\}^\mathbf{N}$ be some periodic symbol. Then, $Per(v)\ne\emptyset$ - and provided $v$ is sufficiently close to $p$, $s\in Per(v)$.
\end{corollary}

\begin{figure}[h]
\centering
\begin{overpic}[width=0.25\textwidth]{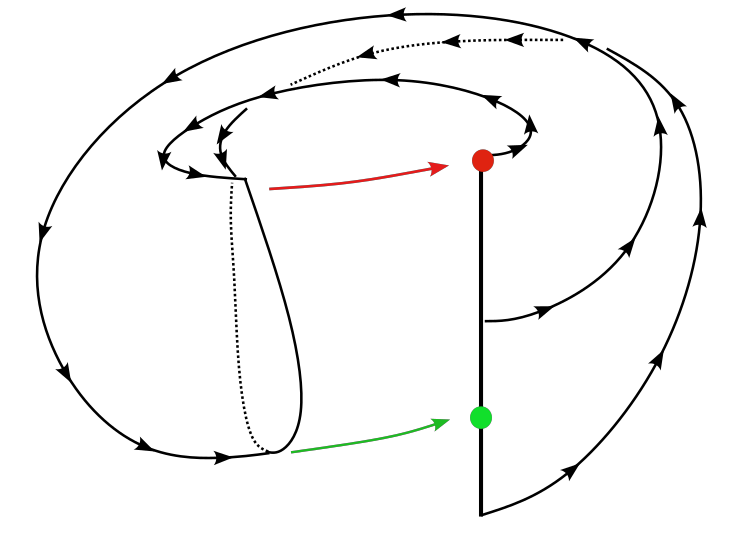}
\end{overpic}
\caption[Fig31]{An illustration of the semiflow created by suspending an interval and folding it into a unimodal map. For an alternative illustration, see Fig.\ref{template}.}
\label{match1}
\end{figure}

In order to give a sketch of proof for Th.\ref{polytheorem}, we now recall some notions from Section $1.2$. Given a parameter $c\in[-2,\frac{1}{4}]$, recall the polynomial $p_c(x)=x^2+c$ there exists a maximal bounded interval $V_c=[x_{2,c},x_{1,c}]$, $x_{2,c}<0<x_{1,c}$ s.t. $p_c(V_c)\subseteq V_c$ (i.e., $p_c$ folds $V_c$ on itself) - and that for $x\not\in V_c$, $p^n_c(x)\to\infty$. Additionally, recall we denote by $I_c$ the maximal invariant set of $p_c$ in $V_c\setminus\{0\}$, and that we defined a symbolic coding by denoting the interval $(0,x_{1,c}]$ with $1$ and $[x_{2,c},0)$ by $2$ - consequentially, there exists a continuous $\xi_c:I_c\to\{1,2\}^\mathbf{N}$ s.t. $\xi_c\circ p_c=\sigma\circ\xi_c$. As an analogue of the set $S_v$, we define $It(c)=\xi_c(I_c)$ - that is, $It(c)$ is the collection of symbolic dynamics generated by $p_c$ on $I_c$. Finally, for $c\in[-2,\frac{1}{4}]$ define $Per(c)\subseteq It(c)$ to be the set of periodic symbols in $It(c)$. By Th.\ref{maximal}, we know that given any periodic $s\in\{1,2\}^\mathbf{N}$, provided $c$ is sufficiently close to $-2$ we have $s\in Per(c)$.\\

Having introduced the sets $S_v,Per(v)$, $It(c)$ and $Per(c)$ for $v\in P$ and $c\in[-2,\frac{1}{4}]$ (respectively), we now give a sketch of proof for Th.\ref{polytheorem}. We will prove Th.\ref{polytheorem} in two stages:
\begin{itemize}
    \item At Stage $I$, we define and analyze the function $\Pi:P\to[-2,\frac{1}{4}]$. In more detail, given $v\in P$, we will define $\Pi(v)$ by considering the supremum over $c\in[-2,\frac{1}{4}]$ s.t.  $Per(v)\subseteq Per(c)$. As we will see, by our definition it will follow $\Pi$ is continuous on regions of structural stability in $P$.
    \item At Stage $II$ we prove that if $\Pi(v)=d$, there exists a factor map $\zeta_v$ between two invariant subsets $J_v,J_d$, for the first-return map $f_v$ and $p_d$ (respectively). Moreover, we will show that given $n>0$, provided the parameter $v$ is sufficiently close to a trefoil parameter $p$, $\zeta_v$ is continuous at least around $n$-distinct periodic orbits in $J_v$ for $f_v$. This will imply Th.\ref{polytheorem}, thus concluding the proof.
\end{itemize}

\begin{figure}[h]
\centering
\begin{overpic}[width=0.25\textwidth]{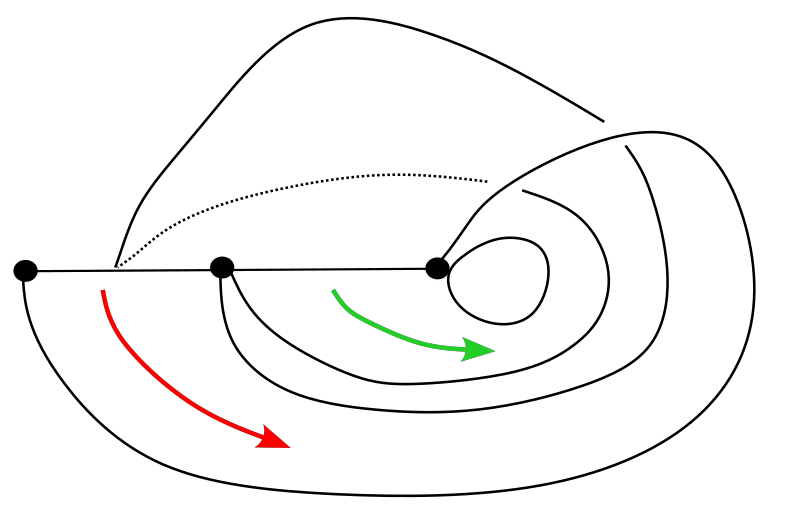}
\end{overpic}
\caption[Fig31]{An alternative illustration of the semiflow created by suspending an interval and folding it into a unimodal map. Here we have a branch line torn in two at the critical point, with the two strands glued together to create a unimodal map.}
\label{template}
\end{figure}

\subsection{\textbf{Stage $I$ - defining $\Pi:O\to[-2,\frac{1}{4}]$:}}

As stated above, we prove Th.\ref{polytheorem} by optimally matching the symbolic dynamics of a given $Per(v)$ with the $Per(d)$ (for some parameters $v\in P$ and $d\in[-2,\frac{1}{4}]$). To do so, given a parameter $v\in P$ first define the function $d: P\to[-2,\frac{1}{4}]$ by $d(v)=d=\sup\{c\in[-2,\frac{1}{4}]|Per(v)\subseteq Per(c)\}$ - since $Per(v)\ne\emptyset$ by Cor.\ref{dista1}, we always have $\frac{1}{4}\geq d(v)\geq-2$. Let us note that when $p\in P$ is a trefoil parameter for the Rössler system, by Cor.\ref{misiu} and Th.\ref{maximal} we have $d(p)=-2$. With these ideas in mind, given $v\in P$ we define $\Pi(v)$ as follows:

\begin{itemize}
    \item For any trefoil parameter $p\in P$, inspired by the discussion above we set $\Pi(p)=-2$.
    \item Assume $v\in P$ is not a trefoil parameter. In that case, set $\Pi(v)=d(v)$.

\end{itemize}

Recalling Cor.\ref{misiu}, we now prove the following Lemma, with which we conclude Stage $I$:
\begin{lemma}
    \label{uniqueness}
Given any $\{v_n\}_n\subseteq P$ s.t. $v_n\to p$, we have $\Pi(v_n)\to -2$. That is, $\Pi$ is continuous at trefoil parameters for the Rössler system.
\end{lemma}
\begin{proof}
To begin, recall that given any periodic $s\in\{1,2\}^\mathbf{N}$, by Th.II.3.2 in \cite{DeMVS} (i.e., Th.\ref{maximal} in Section 1.2) there exists some $c_s\in(-2,\frac{1}{4}]$ s.t. for $c<c_s$, $s\in Per(c)$ - and moreover, the same Theorem also implies $p_{-2}$ is dynamically maximal w.r.t. the quadratic family. That is, given any $c\in[-2,\frac{1}{4}]$ s.t. $Per(c)$ includes every periodic symbol in $\{1,2\}^\mathbf{N}$, we have $c=-2$ - and additionally, when $c\to-2$, the number of distinct periodic symbols in $Per(c)$ increases as well.\\

Now, let $p\in P$ be a trefoil parameter, and consider a sequence $v_n\to p$. By Th.\ref{th41} it follows that for every periodic $s\in\{1,2\}^\mathbf{N}$ and any sequence $\{v_n\}_n$ s.t $v_n\to p$, there exists some $k_s$ s.t. for $n>k_s$, $s\in Per({v_n})$. Now, write $c_s=\sup\{c\in[-2,\frac{1}{4}]|s\in Per(c)\}$ - by the definition of $\Pi$ it now follows that for every $n>k_s$, $-2\leq\Pi(v_n)\leq c_s$. By Th.\ref{maximal} we know $c_s>-2$ - which implies there exists some periodic $\omega\in\{1,2\}^\mathbf{N}$, $\omega\ne s$, s.t. $c_\omega=\sup\{c\in[-2,\frac{1}{4}]|\omega\in Per(c)\}$ satisfies $c_\omega<c_s$. Again, using the same argument, it now follows there exists some $k_\omega$ s.t. for $n>k_\omega$, $\omega\in Per(v_n)$ - which, using similar arguments, again implies that for $n\geq k_\omega$, $-2\leq\Pi(v_n)\leq c_\omega<c_s$.\\

Or, in other words, we have just shown that the infimum $\inf\{c\in[-2,\frac{1}{4}]|\exists k>0,\forall n>k, \Pi(v_n)<c\}$ is precisely $-2$ - consequentially, when $v_n\to p$, $\Pi(v_n)\to-2$ and Lemma \ref{uniqueness} now follows.
\end{proof}
\subsection{\textbf{Stage $II$ - concluding the proof of Th.\ref{polytheorem}:}}
Having defined the function $\Pi$ and analyzed its properties, we are now ready to conclude the proof of Theorem \ref{polytheorem} - namely, we now prove that if $d=\Pi(v)$, $v\in P$, and if $f_v:\overline{U_v}\to \overline{U_v}$ is the first-return map corresponding to the Rössler system at $v$ (wherever defined in $\overline{U_v}$), there exists a factor map between the dynamics of $f_v$ and those of $p_d(x)=x^2+d$ on the invariant interval $V_d$.\\

To begin, first recall we denote by $\sigma:\{1,2\}^\mathbf{N}\to\{1,2\}^\mathbf{N}$ the one-sided shift. Additionally, recall that given a parameter $v\in P$, we denote by $S_v\subseteq\{1,2\}^\mathbf{N}$ the set of all symbols generated by $f_v:\overline{U_v}\to \overline{U_v}$. That is, $S_v=\pi_v(I_v)$, with $\pi_v$ and $I_v\subseteq \overline{U_v}$ as before - i.e., $\pi_v:I_v\to S_v$ is surjective, and satisfies $\pi_v\circ f_v=\sigma\circ\pi_v$ (see Th.\ref{th41} or the discussion immediately before Th.\ref{polytheorem}). Now, recall the coding map $\xi_d:I_d\to\{1,2\}^\mathbf{N}$, $d\in[-2,\frac{1}{4}]$ defined at the end of Section $1.2$ (and in particular, recall the set of itineraries $It(d)=\xi_d(I_d)$ defined at the beginning of Stage $I$) - as shown in Section $1.2$, , $\xi_d:I_d\to It(d)$ is continuous and surjective, and satisfies $\xi_d\circ p_d=\sigma\circ\xi_d$.\\

Now, given $v\in P,d=\Pi(v)$, we define the set $T_v=It(d)\cap S_v$ - or, written differently, $T_v=\xi_d(I_d)\cap\pi_v(I_v)$. That is, $T_v$ is the set of all symbols in $\{1,2\}^\mathbf{N}$ generated by both the first-return map $f_v$ and the quadratic polynomial $p_d$ - it is easy to see $T_v$ is invariant under the one-sided shift $\sigma:\{1,2\}^\mathbf{N}\to\{1,2\}^\mathbf{N}$. Let us remark that by Th.\ref{th41} and the discussion at Section $1.2$ the set $T_v$ is never empty - to see why, recall that by Th.\ref{th41} (and Remark \ref{notemp}) for all $v\in P$ and $d\in[-2,\frac{1}{4}]$, the constant $\{1,1,1...\}$ lies in both $S_v,It(d)$ - consequentially, the constant $\{1,1,1...\}$ lies in $T_v$. Now, define $J_v=\pi^{-1}_v(T_v)$ - by definition, $J_v\subseteq I_v\subseteq \overline{U_v}$, and $J_v$ is $f_v-$invariant. By Th.\ref{th41} and Th.\ref{maximal} we immediately conclude:
\begin{corollary}
    \label{corjf}
    Let $p\in P$ be a trefoil parameter for the Rössler system, let $s\in\{1,2\}^\mathbf{N}$ be periodic of minimal period $k$. Then, provided a parameter $v\in P$ is sufficiently close to $p$ we have the following:
    \begin{itemize}
        \item $s\in T_v$, and $\pi^{-1}_v(s)$ includes $x_s$, a periodic point of minimal period $k$ for $f_v$ (in particular, $x_s\in J_v$).
        \item $f_v$ is continuous at the orbit $\{x_s,...,f^{k-1}(x_s)\}$.
        \item $\pi_v$ is continuous on the orbit $\{x_s,...,f^{k-1}(x_s)\}$.
    \end{itemize}
    
\end{corollary}

Having defined the set $J_v$ for the first-return map $f_v$, we now turn to define the corresponding set $J_d\subseteq I_d$ for $p_d$, where $d=\Pi(v)$. To do so, consider $\xi^{-1}_d(T_v)$ - let $J_d$ denote some $p_d-$invariant set in $\xi^{-1}_d(T_v)$, s.t. $J_d$ satisfies the following three properties:

\begin{itemize}
    \item  $\xi_d(J_d)=T_v$.   
    \item Every component of $J_d$ is a singleton - i.e., $\xi_d:J_d\to T_F$ is a bijection, satisfying $\xi_d\circ p_d=\sigma\circ\xi_d$.
    \item Whenever $s\in T_v$ is periodic of minimal period $k$ for the one-sided shift, $\xi^{-1}_d(s)\cap J_d=\{y_s\}$ is a periodic point of minimal period $k$ for $p_d$.
\end{itemize}

Now, given a parameter $v\in P$, $d=\Pi(v)\in[-2,\frac{1}{4}]$ we define $\zeta_v: J_v\to J_d$ by $\zeta_v=\xi^{-1}_d\circ\pi_v$. Recalling $f_v:\overline{U_v}\to \overline{U_v}$ denotes the first-return map for the Rössler system corresponding to $v$, we immediately conclude:
\begin{corollary}
    \label{corzeta} For any parameter $v\in P$ for the Rössler system, set $d=\Pi(v)$. Then, the function $\zeta_v:J_v\to J_d$ defined above satisfies:

    \begin{itemize}
        \item $\zeta_v$ is surjective.
        \item $\zeta_v\circ f_v=p_d\circ\zeta_v$.

    \end{itemize}
\end{corollary}

In other words, $\zeta_v$ gives us a factor map between the flow dynamics on $J_v$ (or more precisely, the dynamics of the first-return map $f_v$ on $J_v$), and those of the polynomial $p_d$ on its invariant set $J_d$. Since $I_d$ is a subset of a closed, bounded interval (see the discussion in Section $1.2$), so is $J_d$ - i.e., $J_d$ is always bounded. Now, having defined a factor map $\zeta_v:J_v\to J_d$, $d=\Pi(v)$, all that remains to conclude the proof of Th.\ref{polytheorem} is to prove the following: that given any $n>0$, provided the parameter $v$ is sufficiently close to a trefoil parameter $p$, $\zeta_v$ maps $n-$periodic orbits for $f_v$ to $n$ periodic orbits for $p_d$ (without changing their minimal periods) - and moreover, $\zeta_v$ is also continuous at the said $n$ periodic orbits.\\

We will do so in two steps - first, we will prove Lemma \ref{conti2}, where we show that given any $n>0$, provided $v$ is sufficiently close to $p$, then $J_v$ and $J_d$ both include at least $n$ periodic orbits which are mapped on one another by $\zeta_v$. Second, we will prove Lemma \ref{conti} by proving $\zeta_v$ is continuous at least on $n-3$ of those orbits - from which Th.\ref{polytheorem} would follow. We therefore begin with the following fact:
\begin{lemma}
    \label{conti2} 
Let $p\in P$ be a trefoil parameter for the Rössler system. Then, given any $v\in P$ and every collection of distinct, periodic symbols $s_1,...,s_n\in\{1,2\}^\mathbf{N}$ with respective minimal periods $k_1,...,k_n$, provided $v$ is sufficiently close to $p$, setting $d=\Pi(F)$ we have the following:

\begin{itemize}
    \item $s_1,...,s_n\in T_v$.
    \item  $J_v$ contains $n-$periodic orbits for $f_v$ of minimal periods $k_1,...,k_n$, denoted by $\Omega_1,...,\Omega_n$. Moreover, $f_v$ is continuous at $\Omega_1,...,\Omega_n$.
    \item $J_d$ includes $n$ periodic orbits for $p_d$ of minimal periods $k_1,...,k_n$, denoted by $P_1,...,P_n$. 
    \item For every $1\leq i\leq n$, $\zeta_v(\Omega_i)=P_i$.
\end{itemize}
\end{lemma}
\begin{proof}

Let $s_1,...,s_n$ and $k_1,...,k_n$ be as above. To begin, recall that by Th.\ref{th41} and Cor.\ref{corjf}, whenever $v$ is sufficiently close to $p$ we have $s_1,..,s_n\in S_v$ - and consequentially, for every $1\leq i\leq n$, $\pi^{-1}_v(s_i)$ includes a periodic point $x_i$ for $f_v$ of minimal period $k_i$, $1\leq i\leq n$. To continue, set $\Omega_i$, $1\leq i\leq n$ as the orbit of $x_i$ - by Th.\ref{th41}, provided $v$ is sufficiently close to $p$, $f_v$ is continuous at $\Omega_i$. Now, set $d=\Pi(v)$ and recall that when $v\to p$ we have $\Pi(F)\to-2$ (see Lemma \ref{uniqueness}) - therefore, provided $v$ is sufficiently close to $p$, by Lemma \ref{uniqueness} and Th.\ref{maximal} we have $s_1,...,s_n\in It(d)$.\\

This implies that whenever $v$ and $p$ are sufficiently close, by $T_v=\xi_d(I_d)\cap \pi_v(I_v)$ we have $s_1,...,s_n\in T_v$. As such, by the definition of $J_d$ above, this implies there exist periodic orbits $P_1,...,P_n\subseteq J_d$, periodic orbits for $p_d$ of minimal periods $k_1,...,k_n$, s.t. $\xi_d(P_i)=s_i$ (for $1\leq i\leq n$). Therefore, all in all, by the definition of the factor map $\zeta_v$ we conclude $\zeta_v(\Omega_i)=P_i$, $1\leq i\leq n$ and Lemma \ref{conti2} now follows.
\end{proof}
Having proven Lemma \ref{conti2}, we are almost done proving Th.\ref{polytheorem} - with the only part remaining is proving the increasing continuity $\zeta_v$ around periodic orbits in $J_v$ as $v\to p$. We we will do so using the notions of the (real) Julia and Fatou sets introduced at Section $1.2$. To this end, recall we denote by $\sigma:\{1,2\}^\mathbf{N}\to\{1,2\}^\mathbf{N}$ the one sided shift. We now prove:

\begin{lemma}
    \label{conti}
Let $p\in P$ be a trefoil parameter for the Rössler system, and choose some $n>3$. Then, provided $v\in P$ is sufficiently close to $p$ there exist at least $n-3$ distinct periodic orbits $\Omega_1,...,\Omega_{n-3}\subseteq J_v$ at which the factor map $\zeta_v$ is continuous.

\end{lemma}
\begin{proof}

Let $s_1,...s_n\in\{1,2\}^\mathbf{N}$ be periodic symbols s.t. each $s_i$, $1\leq i\leq n$ lies on a different periodic orbit for the one-sided shift $\sigma$. Now, assume $v$ is sufficiently close to $p$ s.t. Lemma \ref{conti2} holds - that is, w.r.t. $d=\Pi(v)$, $\zeta_v=\xi^{-1}_d\circ\pi_v$ we have:
\begin{itemize}
    \item $s_1,...,s_n\in T_v$.
    \item  $J_v$ contains $n$ distinct periodic orbits for $f_v$, denoted by $\Omega_1,...,\Omega_n$ - by the argument presented in the proof of Lemma \ref{conti2}, we know $s_i\in\pi_v(\Omega_i)$, $i=1,...,n$.
    \item $J_d$ includes $n$ distinct periodic orbits for $p_d$, denoted by $P_1,...,P_n$, s.t. $\zeta_v(\Omega_i)=P_i$, $1\leq i\leq n$ - moreover, $P_i$ and $\Omega_i$ have the same minimal periods. Moreover, $s_i\in\xi_d(P_i)$.
\end{itemize}

By Th.\ref{th41}, provided $v$ is sufficiently close to $p$, the map $\pi_v:J_v\to T_v$ is continuous on points on the orbits $\Omega_1,...,\Omega_n$. Therefore, by $\zeta_v=\xi_d^{-1}\circ\pi_v$, it would suffice to show that whenever $v$ is sufficiently close to $p$, $\xi^{-1}_d:T_v\to J_d$ is continuous at $\xi_d(\cup_{i=1}^{n-3}P_i)$. To do so, recall that by Th.\ref{classification}, we already know that, without any loss of generality, $P_1,...,P_{n-1}$ are all repelling and lie in the real Julia set of $p_d$ (see Def.\ref{julia}). Now, consider the Fatou set of $p_d$, which lies in the invariant interval $V_d$ (see Def.\ref{julia}) - which we denote $G_d$. By Th.\ref{classification}, it is an open set which includes at most one periodic orbit, say, $O_1$ - therefore, $O_1$ lies at some finite collection of components of $G_d$, which we denote by $C_1$. Since $O_1\subseteq G_d$ and since $P_1,...,P_{n-1}$ are all at the Julia set, it is easy to see $O_1\ne P_i$, $1\leq i\leq n-1$ (see the illustration in Fig.\ref{conti1}).\\

Now, recall that by the No Wandering Domain Theorem (see \cite{Su} or Theorem VI.A in in \cite{DeMVS}), given any component $C_2\subseteq G_d$ s.t. $C_2\cap C_1=\emptyset$, there exists some $k$ s.t. $p^k_d(\overline{C})=\overline{C_1}$. Consequentially, there are no periodic orbits for $p_d$ which intersect $\partial C_2$. Similarly, it follows $\partial C_1$ may include at most two more periodic orbits for $p_d$, say $O_2$ and $O_3$, both of which lie on the real Julia set - in particular, we conclude $\overline{G_d}$ includes at most three periodic orbits. Therefore, it follows that without any loss of generality, for every $1\leq i\leq n-3$, $P_i$ lies away from $\overline{G_d}$ - and in particular, $P_i\cap\partial G_d=\emptyset$.\\

\begin{figure}[h]
\centering
\begin{overpic}[width=0.5\textwidth]{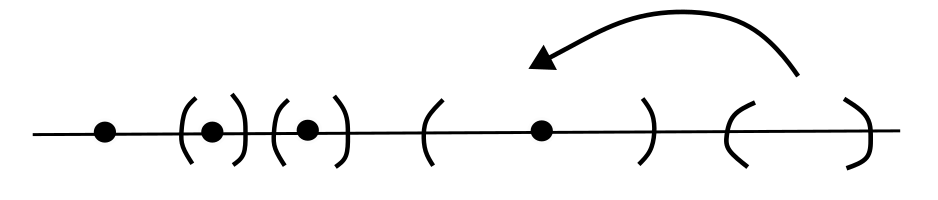}
\put(200,0){$x_{n+1}$}
\put(300,0){$x_n$}
\put(450,0){$C_1$}
\put(550,0){$O_1$}
\put(830,0){$C_2$}
\put(100,0){$x$}
\end{overpic}
\caption[Fig31]{The dynamics of $p_d$ in $V_d$ - $C_2$ is a component in $G_d$ which is eventually mapped on $C_1$ ($C_1$ includes a periodic point $O_1$ for $p_d$). $x=\xi^{-1}_d(s)$, while $x_n=\xi^{-1}_d(s_n)$, for some $\{s_n\}_n\subseteq T_F$, $s_n\to s$. In particular, $x_n$ lie in components of $I_d$ which become smaller and smaller as $s_n\to s$ - consequentially, since $\xi^{-1}_d(s)=\{x\}$, it follows $x_n\to x$. }
\label{conti1}
\end{figure}
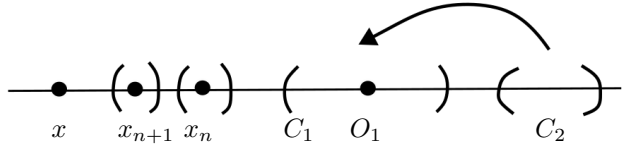

To continue, recall $I_d$ denotes the maximal invariant set of $p_d$ in $V_d\setminus\{0\}$ (see the end of Section $1.2$ for more details). As a consequence from previous paragraph, we conclude that given $x\in\cup_{i=1}^{n-3}P_i$, $x$ is a point in the real Julia set which lies away from $\overline{G_d}$. Therefore, writing $\xi_d(x)=s$, by Cor.\ref{single} it follows $\xi^{-1}_d(s)$ is a singleton in $I_d$ - which implies that given any $s_n\to s$ in $T_v$, $\xi^{-1}_d(s_n)=x_n$ must tend to $x$, i.e., $\xi^{-1}_d$ is continuous at $s$ (see the illustration in Fig.\ref{conti1}). Since $x\in\cup_{i=1}^{n-3}P_i$ was chosen arbitrarily, it follows $\xi^{-1}_d:T_v\to J_d$ is continuous on $\xi_d(\cup_{i=1}^{n-3}P_i)$. Therefore, by previous discussion and by $\zeta_v=\xi^{-1}_d\circ\pi_v$ we conclude $\zeta_v$ is continuous at $\cup_{i=1}^{n-3}P_i$ and Lemma \ref{conti} now follows.
\end{proof}
Having proven Lemma \ref{conti}, we now conclude the proof of Th.\ref{polytheorem}. Summarizing our results, we have proven the existence of a function $\Pi:P\to[-2,\frac{1}{4}]$ s.t. given $v\in P$, $d=\Pi(v)$ the following is satisfied:

\begin{itemize}
    \item By Lemma \ref{uniqueness}, $\Pi$ is continuous at trefoil parameters.
    \item By Cor.\ref{corjf} and the discussion preceding Lemma \ref{conti2}, there exists an $f_v-$invariant $J_v\subseteq \overline{U_v}$, a bounded, $p_d-$invariant $J_d\subseteq\mathbf{R}$, and a function $\zeta_v:J_v\to J_d$ s.t. $\zeta_v\circ f_v=p_d\circ\zeta_v$.
    \item Given any $n>0$, by Lemmas \ref{conti} and Lemma \ref{conti2}, provided $v$ is sufficiently close to a trefoil parameter $p$, $J_v$ includes at least $n$ distinct periodic orbits $\Omega_1,...,\Omega_n$ for the first-return map $f_v$. Moreover, both $f_v$ and $\zeta_v$ are continuous at $\Omega_1,...,\Omega_n$.
    \item For every $1\leq i\leq n$, $P_i=\zeta_v(\Omega_i)$ is a periodic point for $p_d$ of the same minimal period as $\Omega_i$.
\end{itemize}

The proof of Th.\ref{polytheorem} is now complete.
\end{proof}
\begin{remark}
  Let us note Th.\ref{polytheorem} does not rule out the possibility $\Pi$ is in fact constant throughout the parameter space $P$ - however, in light of the rich bifurcation structure that was observed numerically for the Rössler system, this is very unlikely to be the case (see, for example, \cite{MBKPS}, \cite{BBS} and \cite{G}).
\end{remark}
\begin{remark}
Let $p$ be a trefoil parameter for the Rössler system, and assume $v=(a,b,c)$ is s.t. the vector field $F_v$ (see Eq.\ref{Field}) generates an attractor $A$.  In conjunction with Th.\ref{attrac}, Th.\ref{polytheorem} implies the following heuristic about the dynamics on the Rössler attractor - provided $v$ is sufficiently close to $p$, and provided the flow is sufficiently contracting around $A$, Th.\ref{polytheorem} states the dynamics on $A$ are essentially those of a suspended quadratic polynomial.
\end{remark}

\section{Discussion}
Before we conclude this paper, let us discuss how the Rössler system develops from order to chaos, and how both Th.\ref{attrac} and Th.\ref{polytheorem} relate to it. To motivate this discussion, let us remark that even though Th.\ref{polytheorem} gives a partial explanation why the "seemingly polynomial" behavior of the Rössler system is something to be expected, it does not provide us with clear topological mechanism explaining the appearance of such phenomena. We therefore conclude this paper by proposing such a mechanism, inspired by both Th.\ref{polytheorem}, the numerical studies (in particular, \cite{MBKPS} and \cite{Le}), and the Kneading Invariant introduced in Section $1.2$.\\

To begin, recall the component $\Delta_{Out}$ of the one-dimensional unstable manifold $W^u_{Out}$ introduced at the beginning of Section $2$. Additionally, recall that if $p\in P$ is a trefoil parameter for the Rössler system (see Def.\ref{def32}), when $v\to P$, $v\in P$, the invariant manifold $\Delta_{Out}$ is continuously deformed to the heteroclinic trajectory $\Theta$ (see Fig.\ref{fig7}). Now, let $v\in P$ be a parameter s.t. $\Delta_{Out}$ is bounded for $F_v$ -  by Cor.\ref{TR} it follows that for any such $v$, the separatrix $\Delta_{Out}$ either limits to a fixed point, or intersects with the cross-section $U_v$ transversely infinitely many times.\\

This motivates us to define a sequence $\{x_n\}_{n\geq0}$ as follows: let $p$ denote the first intersection point of the separatrix $\Delta_{Out}$ with $U_v$ (see Fig.\ref{deform5}), and define the sequence $\{x_n\}_{n\geq0}\subseteq (\Delta_{Out}\cap \overline{U_v})\cup\{P_{In},P_{Out}\}$ as follows - $x_0=P_{Out}$, $x_1=p$, and from then onward $x_{n+1}=f_v(x_n)$. In the particular case where $\Delta_{Out}$ is a heteroclinic trajectory which intersects with $\overline{U_v}$ a finite number of times, say, $\Delta_{Out}\cap\overline{U_v}=\{x_0,...,x_k\}$, we set $x_{k+j}=P_{In},j\geq1$ (see the illustration in Fig.\ref{deform5}).\\
\begin{figure}[h]
\centering
\begin{overpic}[width=0.3\textwidth]{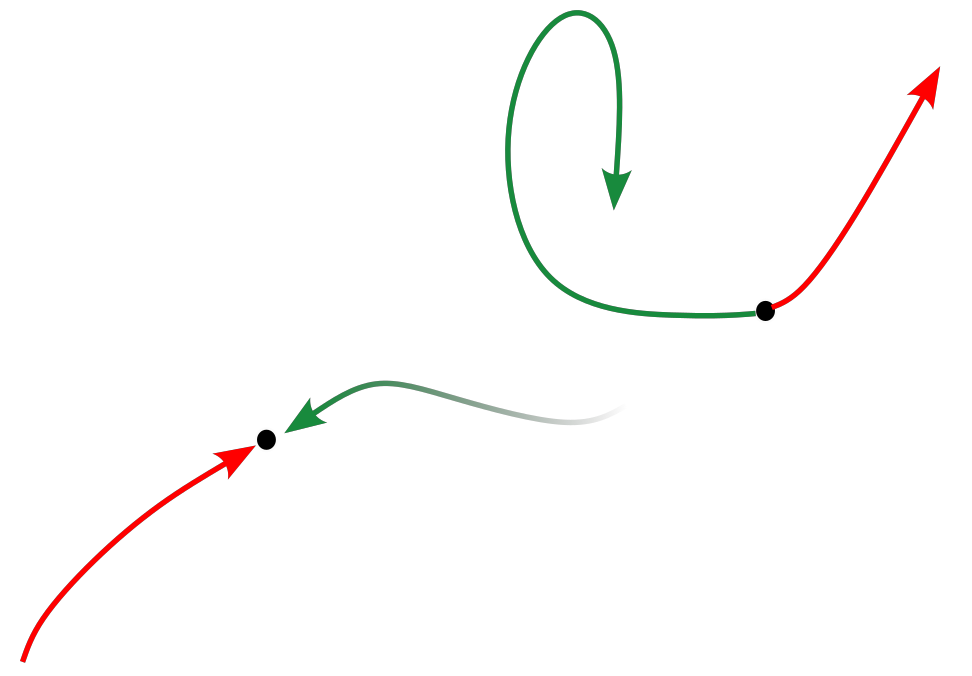}
\put(430,200){$\Delta_{In}$}
\put(300,190){$P_{In}$}
\put(50,220){$\Gamma_{In}$}
\put(750,520){$\Gamma_{Out}$}
\put(750,290){$P_{Out}$}
\put(560,450){$\Delta_{Out}$}
\end{overpic}
\caption[Fig16]{The separatrices $\Delta_{In},\Delta_{Out}$. The separatrices $\Gamma_{In},\Gamma_{Out}$ connect the fixed points to $\infty$ (see Th.\ref{th21}).}
\label{deform3}
\end{figure}

Now, recall we partitioned the cross-section $\overline{U_v}$ to two sets, $D_1,D_2$ - see the discussion preceding Th.\ref{th41} and the illustration in Fig.\ref{part}. Now, define the sequence $K(v)\in\{1,2\}^\mathbf{N}$, $K(v)=\{s_0,s_1,s_2,...\}$ by $s_n=1$ when $x_n\in D_1$ and $s_n=2$ otherwise. Similarly to the one-dimensional kneading invariant, $K(v)$ describes the symbolic dynamics of $\Delta_{Out}$ - i.e., it describes how $\Delta_{Out}$ intersects with the cross-section $U_v$, which constrains the possible flow dynamics in $\mathbf{R}^3\setminus (\Delta_{Out}\cup\{P_{In},P_{Out}\})$. Therefore, inspired by Th.\ref{maximal}, the proof of Th.\ref{attrac} and Th\ref{polytheorem}, we are led to the following conjecture:

\begin{conj}
\label{polyconj}Let $p\in P$ be a trefoil parameter for the Rössler system (see Eq.\ref{Field} and Def.\ref{def32}). Then, there exists a positive $\epsilon>0$ s.t. for every parameter $(a,b,c)=v\in \{w\in P|||w-p||<\epsilon\}$ the sequence $K(v)$ is well-defined and satisfies:

 \begin{itemize}
     \item If $K(v)=K(w)$, the vector fields $F_v,F_w$ define orbitally equivalent flows in $\mathbf{R}^3$ - that is, $K(v)$ completely determines the dynamics of $F_v$. 
     \item There exists a parameter $c\in[-2,\frac{1}{4}]$ s.t. $K(v)$ is the kneading invariant for $p_c(x)=x^2+c$ - that is, the correspondence $v\to K(v)$ defines a map from a neighborhood of $p$ in $P$ to $[-2,\frac{1}{4}]$. Moreover, the map $v\to K(v)$ is continuous and non-constant.
     \item Let $p_c(x)=x^2+c$ be a polynomial with a kneading invariant $K(v)$ and an invariant interval $V_c$ (see Section $1.2$) - then, there exists a continuous surjection $\zeta_v:\overline{U_v}\to V_c$ s.t. $\zeta_v\circ f_v=p_c\circ\zeta_v$.
     \item If $K(v)=\{1,2,1,1,1...\}$, then $v$ is a trefoil parameter - that is, the dynamics of the Rössler system at trefoil parameter are dynamically maximal.
     \item The function $v\to K(v)$ is continuous on structurally stable sets. Additionally, when $v\to p$, $K(v)\to\{1,2,1,1,1,...\}$
 \end{itemize}
\end{conj}
\begin{figure}[h]
\centering
\begin{overpic}[width=0.45\textwidth]{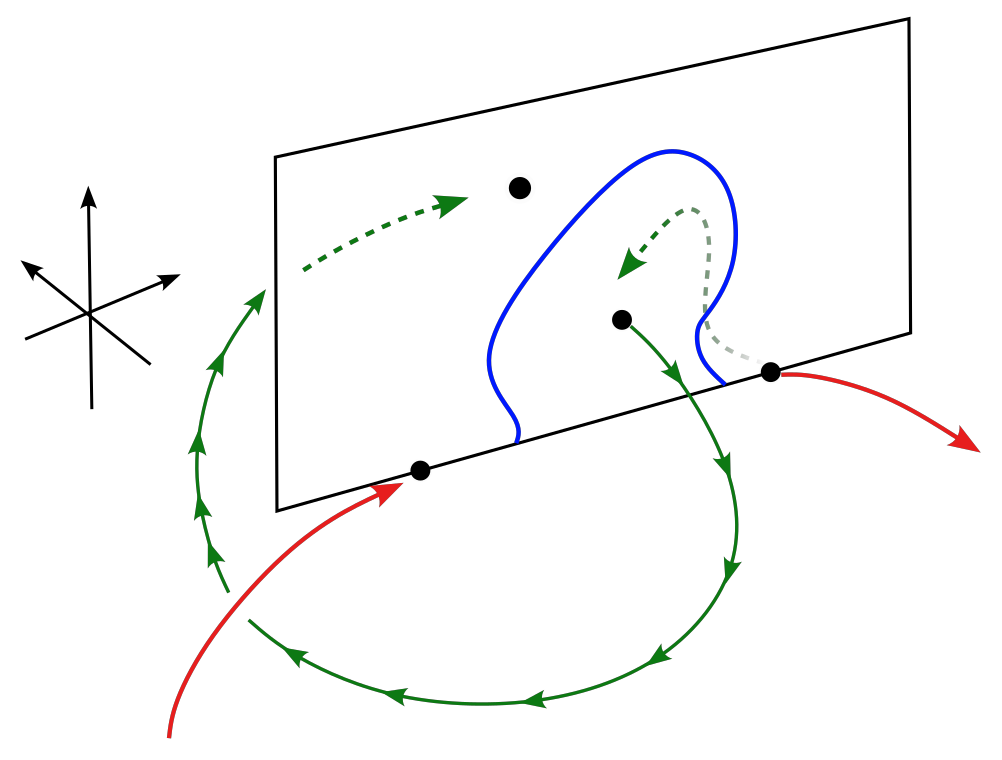}
\put(75,580){$z$}
\put(320,550){$D_1$}
\put(465,520){$x_2$}
\put(440,430){$\rho$}
\put(570,470){$x_1$}
\put(620,550){$D_2$}
\put(800,480){$U_p$}
\put(190,490){$x$}
\put(0,510){$y$}
\put(455,90){$\Delta_{Out}$}
\put(425,240){$P_{In}$}
\put(200,50){$\Gamma_{In}$}
\put(750,330){$P_{Out}$}
\put(910,360){$\Gamma_{Out}$}
\end{overpic}
\caption[Fig31]{The trajectory of $\Delta_{Out}$, connecting $x_n$ to $x_{n+1},n\geq0$ (with $x_0=P_{Out}$).}
\label{deform5}
\end{figure}

In other words, Conjecture \ref{polyconj} implies that at least around trefoil parameters, the dynamical complexity of the Rössler system is completely determined by the invariant manifold $\Delta_{Out}$. Moreover, it also implies $\Delta_{Out}$ constrains the possible flow dynamics similarly to how the orbit of the critical point constrains the dynamics of a quadratic polynomial $p_c$. In addition, it also implies that continuous families of one-dimensional maps which fold the space (like, say, the Quadratic family or the Tent family) serve as singular models for the dynamics of the Rössler attractor - and more generally, of other $C^1$ families of vector fields whose dynamics stretch and fold $\mathbf{R}^3$.\\

In this spirit, before concluding this paper let us remark there exists a generalization of the kneading invariant to two-dimensional diffeomorphisms - the Braid Type (see \cite{CH}). In \cite{CH}, the Braid Type was applied to study the dynamics of the Henon map (see \cite{He}). The Henon map, originally introduced in \cite{He} as a diffeomorphism which folds the plane, is known numerically to have a similar bifurcation diagram to that of the Rössler system - for example, one can compare the spiral structure observed for the Rössler system in \cite{G} with the bifurcation diagram of the Henon map in \cite{G1}. Of course, it is probably impossible to directly apply Braid Types to study the dynamics of the Rössler system - if only because the is no reason to assume the first-return map for the flow is even globally defined. However, due to the similarities cited above, it may well be possible the dynamics of the Rössler system and those of the Henon map are somehow related. Due to the well-known connection between the Henon map and the dynamics of quadratic polynomials, we conjecture further study of Conj.\ref{polyconj} could at least partially explain this connection.

%
%
%
\printbibliography

@article{Ross76,
        author = "Rössler, O.E.",
        title = "An equation for continuous chaos",
        journal = "Physics Letters A",
        volume = "57",
        year = "1976",
        pages = "397-398",
}

@article{Zgli97,
        author = "Zgliczynski, P.",
        title = "Computer assisted proof of chaos in the Rössler equations and in the H´enon map",
          journal = "Nonlinearity",
        volume = "10(1)",
        year = "1997",
        pages = "243-252",
}

@article{XSYS03,
        author = "Yang, X.S. and Yu Y., and Zhang S.",
        title = " A new proof for existence of horseshoe in the Rössler system",
          journal = "Chaos, Solitons, and Fractals",
        volume = "18",
        year = "2003",
        pages = "223-227",
}

@article{S,
        author = "Smale, S.",
        title = "Differentiable dynamical systems",
          journal = "Bull. Amer. Math. Soc.",
        volume = "73",
        year = "1967",
        pages = "747-817",
}

@article{LeS,
        author = "Shilnikov, L.",
        title = "A case of the existence of a denumerable set of periodic motions",
          journal = "Sov. Math. Dok.",
        volume = "6",
        year = "1967",
        pages = "163-166",
}

@article{Ro83,
        author = "Rössler, O.E.",
        title = "The Chaotic Hierarchy",
          journal = "Zeitschrift für Naturforschung A",
        volume = "38",
        year = "1983",
  }

@article{MBKPS,
        author = "Malykh, S., and Bakhanova, Y., and Kazakov, A., and Pusuluri, K., and Shilnikov, A.",
        title = "Homoclinic chaos in the Rössler model",
          journal = "Chaos",
        volume = "30",
        year = "2020",
}

@article{BBS,
        author = "Barrio, R., and Blesa, F. and Serrano, S.",
        title = "Topological Changes in Periodicity Hubs of Dissipative Systems",
          journal = "Phys. Rev. Lett.",
        volume = "108, 214102",
        year = "2012",
}

@article{Y,
        author = "Yorke, J.A., and Alligood, K.T.",
        title = "Period Doubling Cascades of Attractors: A Prerequisite for Horseshoes",
          journal = "Communications in mathematical physics",
        volume = "101",
        year = "1985",
        pages = "305-321",
}

@article{zgli1,
        author = "Gierzkiewicz, A., and Zgliczyński, P.",
        title = "From the Sharkovskii theorem to periodic orbits for the Rössler system",
          journal = "Journal of Differential Equations",
        volume = "314",
        year = "2022",
        pages = "733-751",
}

@article{zgli2,
        author = "Gierzkiewicz, A., and Zgliczyński, P.",
        title = "Periodic orbits in the Rössler system",
          journal = "Communications in Nonlinear Science and Numerical Simulation
",
        volume = "101",
        year = "2021",
}

@article{Fei,
        author = "Feigenbaum, M.J.",
        title = "Universality in complex discrete dynamics",
          journal = "Los Alamos Theoretical Division Annual Report",
        volume = "1975-1976",
}

@article{G,
        author = "Gallas, J.C.",
        title = "The Structure of Infinite Periodic and Chaotic Hub Cascades in Phase Diagrams of Simple Autonomous Flows",
          journal = "International Journal of Bifurcation and Chaos",
        volume = "20(2)",
        year = "2010",
        pages = "197-211"
}

@article{GKP,
        author = "Gaspard, P., and Kapral, R., and Nicolis, G.",
        title = "Bifurcation Phenomena near Homoclinic Systems: 
A Two-Parameter Analysis",
          journal = "Journal of Statistical Physics",
        volume = "35 (5/6)",
        year = "1984",
        pages = "597-727"
}

@article{BB2,
        author = "Barrio, R., and Blesa, F., and Serrano, S.",
        title = "Unbounded dynamics in dissipative flows: Rössler model",
          journal = "Chaos",
        volume = "242",
        year = "2014",
}

@article{Le,
        author = "Letellier, C., and Dutertre, P., and Maheu, B.",
        title = "Unstable periodic orbits and templates of the Rössler system: Toward a systematic topological characterization",
          journal = " Chaos",
        volume = "5, 271",
        year = "1995",
}

@article{WZ,
        author = "Wilczak, D., and Zglizcynski, P.",
        title = "Period Doubling in the Rössler System—A Computer Assisted Proof",
          journal = "Foundations of Computational Mathematics",
        volume = "9",
        year = "2009",
        pages = "611-649",
}

@article{Su,
        author = "Sullivan, D.",
        title = "Quasiconformal homeomorphisms and dynamics. I. Solution of the Fatou-Julia problem on wandering domains",
          journal = " Ann. of Math.",
        volume = "(2) 122 no.3",
        year = "1985",
        pages = "401-418",
}

@book{DeMVS,
  author = "De-Melo, W., and Van Strien, S.J.",
  year = "1993",
  title = "One-Dimensional Dynamics",
  publisher = "Springer"
}

@article{Mis,
        author = "Misiurewicz, M.",
        title = "Absolutely continuous measures for certain maps of an interval",
          journal = "Mathématiques de L’Institut des Hautes Scientifiques",
        volume = "53",
        year = "1981",
        pages = "17-51",
}

@article{Pan,
        author = "Teryokhin, M.T., and Paniflova, T.L.",
        title = "Periodic Solutions of the Rössler System",
          journal = "Russian Mathematics",
        volume = "43 (8)",
        year = "1999",
        pages = "66-69",
}

@article{CNV,
        author = "Cândido, M.R., and D. Novaes, D., and Valls, C.",
        title = "Periodic solutions and invariant torus in the Rössler system",
          journal = "Nonlinearity",
        volume = "33 4512",
        year = "2020",
        pages = "66-69",
}

@article{B,
        author = "Benini, A.M.",
        title = "A survey on MLC, Rigidity, and related topics",
          journal = "arXiv:1709.09869",
         year = "2017",
  }

@book{SSTC,
  author = "Chua, L.O., and Shilnikov, L.P., and Shilnikov, A. and Turaev, D.V.",
  year = "2001",
  title = "Methods of Qualitative Theory in Nonlinear Dynamics, Part II",
  publisher = "New Jersey: World Scientific"
}

@inbook{SR,
    author = "Barrio, R., and Shilnikov, A., and Shilnikov, L.P.",
    title = {Chaos, CNN, Memristors and beyond – a festschrift for Leon Chua},
    publisher = "World Scientific",
    year = "2013",
    chapter = "33 - \textit{Symbolic Dynamics and Spiral Structures due to the Saddle Focus Bifurcations}",
}

@article{ST,
        author = "van Strien, S.J.",
        title = "On the bifurcations creating horseshoes",
          journal = "Dynamical Systems and Turbulence, Warwick 1980, Proceedings of a Symposium Held at the University of Warwick 1979/80",
        volume = ", vol. 898 of Lecture Notes in Math.",
        year = "1981"
}

@article{CH,
        author = "Carvalho, A.D. and Hall, T.",
        title = "How to prune a horseshoe",
          journal = "Nonlinearity",
        volume = "15 R19",
        year = "2002"
}

@article{BeH,
        author = "Betsvina, M. and Handel, M.",
        title = "Train-tracks for surface homeomorphisms",
          journal = "Topology",
        volume = "34 (1)",
        year = "1995",
        pages = "109-140"
}

@article{Pi,
    author = "Pinsky, T.",
    title ="Analytical study of the Lorenz system: Existence of infinitely many periodic orbits and their topological characterization
" ,
    journal ="Proceedings of the National Academy of Sciences" ,
    volume = "120",
    year ="2023"
}

@article{PY,
        author = "Mallet-Paret, J., and Yorke, J.A.",
        title = "Snakes: Oriented Families of Periodic Orbits, Their Sources, Sinks, and Continuation",
          journal = "Journal of Differential Equations",
        volume = "43",
        year = "1982",
        pages = "419-450"
}

@article{SBM,
        author = "Kr. Sarmah, Hemanta., and Kr. Baishya, Tapan, and Ch. Das, Mridul",
        title = "Period Doubling and Feigenbaum Universality in the Rössler system",
          journal = "Journal of Global Research in Mathematical Archives",
        volume = "1",
        year = "2013",
}

@article{Chu,
        author = "Churchill, R.C.",
        title = "Isolated Invariant Sets in Compact Metric Spaces",
          journal = "Journal of Differential Equations",
        volume = "12",
        year = "1972",
        pages = "330-352"
}

@book{CG,
  author = "Carleson, L., and Gsmelin, T.W.",
  year = "1992",
  title = "Complex Dynamics",
  publisher = "Springer"
}

@online{I,
    author = "Igra, E.",
    title = "Knots and Chaos in the Rössler system",
    url  = "https://arxiv.org/abs/2306.04772",
}

@online{I2,
    author = "Igra, E.",
    title = "Topological lower bounds for the Rössler system",
    url  = "https://arxiv.org/abs/2310.14066",
}

@book{GN,
  author = "Gidea, M. and Niculescu, C.P.",
  year = "2002",
  title = "{Chaotic Dynamical Systems: An Introduction}",
  publisher = "Universitaria Press, Craiova"
}

@article{He,
        author = "Henon, M.",
        title = "A two-dimensional mapping with a strange attractor",
          journal = "Comm. Math. Phys.",
        volume = "50",
        year = "1976",
        pages = "69-77"
}

@article{LiLl,
        author = "Lima, M.F.S., and Llibre, J.",
        title = "Global dynamics of the Rössler system with
conserved quantities",
          journal = "J. Phys. A: Math. Theor. ",
        volume = "44",
        year = "2011"
}

@book{GL,
  author = "Gilmore, R., and Lefranc, M.",
  year = "2003",
  title = "The Topology of Chaos: Alice in Stretch and Squeezeland",
  publisher = "John Wiley and Sons, INC., New York"
}

@book{VS,
  author = "Soltan, V.",
  year = "2015",
  title = "Lectures on Convex Sets",
  publisher = "World Scientific"
}

@article{G1,
        author = "Gallas, J.A.C.",
        title = "Structure of the parameter space of the Hénon map",
          journal = "Phys. Rev. Lett.",
        volume = "70",
        year = "1993",
}
\end{document}